\documentclass[12pt,oneside]{article}
\usepackage{amssymb,amsmath,amsfonts,amsbsy, amsthm, xspace, latexsym, amscd, graphicx, fancyhdr,setspace,url}
\usepackage[margin=1in]{geometry}
\usepackage{fancyhdr}
\usepackage{enumitem}

\swapnumbers
\theoremstyle{definition}
\newtheorem{definition}{Definition}[section]
\newtheorem{remark}[definition]{Remark}
\newtheorem{example}[definition]{Example}

\newtheorem{theorem}[definition]{Theorem}
\newtheorem{proposition}[definition]{Proposition}
\newtheorem{lemma}[definition]{Lemma}
\newtheorem{corollary}[definition]{Corollary}

\newcommand{\ifff}{\emph{ i\,f\,f }\;} 

\newcommand{\supp}{\ensuremath{{\rm{supp}}}}
\newcommand{\card}{\ensuremath{{\rm{card}}}}

\newcommand{\ub}{\ensuremath{\mathcal{UB}}}

\newlist{def-enumerate}{enumerate}{2}
\newlist{thm-enumerate}{enumerate}{2}
\newlist{set-enum}{enumerate}{1}
\newlist{axiom-enum}{enumerate}{2}

\setlist[def-enumerate,1]{label=\textbf{\arabic*}.,leftmargin=*,labelindent=\parindent}
\setlist[def-enumerate,2]{label=\textbf{(\alph*)},leftmargin=*,labelindent=\parindent}
\setlist[thm-enumerate]{label=(\emph{\roman*}),leftmargin=*,labelindent=\parindent,align=right}
\setlist[set-enum]{label=\textbf{Set \arabic*},leftmargin=*}
\setlist[axiom-enum,1]{label=\textbf{Axiom \arabic*}, leftmargin=*}
\setlist[axiom-enum,2]{label=\textbf{Axiom (\theenumi.\arabic*)}, leftmargin=*}

\begin{document}

\thispagestyle{empty}
\pagenumbering{roman}
\begin{center}
  \mbox{}
  \vspace{10mm}\\
  {\Huge Completeness of Ordered Fields}\\
  \vspace{10mm}
  {\large By }\\
  {\Large
    \vspace{4mm}
    James Forsythe Hall
  }\\
  \vspace{100mm}
  {\small
    SUBMITTED IN PARTIAL FULFILLMENT OF THE REQUIREMENTS\\
    FOR THE DEGREE OF BACHELOR OF SCIENCE IN MATHEMATICS\\
    AT
    CALIFORNIA POLYTECHNIC STATE UNIVERSITY\\
    SAN LUIS OBISPO\\
    DECEMBER 2010
  }
  \\
  {\scriptsize \copyright 2010 James Hall}
\end{center}
  \vspace{10mm}
  \mbox{}
  {\renewcommand{\thefootnote}{}
    \footnotetext{\scriptsize{This is a senior project done under the supervision of Todor D. Todorov}}
  }

\newpage
\onehalfspacing
\thispagestyle{plain}


\newpage
\thispagestyle{plain}
\section*{Introduction}
In most textbooks, the set of real numbers $\mathbb{R}$ is commonly taken to be a totally ordered Dedekind complete field. Following from this definition, one can then establish the basic properties of $\mathbb{R}$ such as the Bolzano-Weierstrass property, the Monotone Convergence property, the Cantor completeness of $\mathbb{R}$ (Definition~\ref{D: completeness}), and the sequential (Cauchy) completeness of $\mathbb{R}$. The main goal of this project is to establish the equivalence of the preceding properties, in the setting of a totally ordered Archimedean field (Theorem~\ref{T: BIG ONE}), along with a less well known algebraic form of completeness (Hilbert completeness, Definition~\ref{D: completeness}) and a property from non-standard analysis which roughly states that every finite element from the non-standard extension of the field is ``infinitely close'' to some element from the field (Leibniz completeness, Definition~\ref{D: completeness}). (The phrase infinitely close may be off-putting to some as it is sometimes associated with mathematics lacking in rigour, but in section~\S\ref{S: inf} we properly define this terminology and in section~\S\ref{S: examples 1} and \S\ref{S: examples 2} we provide several examples of fields with non-trivial infinitesimal elements.)

As is usual in mathematics, we continued our research past that of Archimedean fields to determine how removing the assumption that the field is Archimedean affects the equivalency of the properties listed in Theorem~\ref{T: BIG ONE}. What we found is that all but three of those properties are false in non-Archimedean fields: Hilbert completeness, Cantor completeness and sequential completeness. Furthermore, we found that the last two of these three properties are no longer equivalent; rather, the latter is a necessary, but not sufficient, condition for the former (see Theorem~\ref{T: non-arch completeness} and Example~\ref{E: robinson asymptotic}).

One application of Theorem~\ref{T: BIG ONE} is given in section~\S\ref{S: axioms}, where we establish eight different equivalent collections of axioms, each of which can be used as axiomatic definitions of the set of real numbers. Another application 
is an alternative perspective of classical mathematics that results from the equivalency of Dedekind completeness (Definition~\ref{D: completeness}) and the non-standard property outlined in the first paragraph above (this property is defined in Definition~\ref{D: completeness}).



\onehalfspacing
\pagenumbering{arabic}
\pagestyle{fancy}
\renewcommand{\sectionmark}[1]{\markright{\thesection.\ #1}}
\fancyhf{}
\fancyhead[L]{\rightmark}
\fancyhead[R]{\bfseries \thepage}
\renewcommand{\headrulewidth}{0.0pt} 
\setcounter{page}{1}

\section{Orderable Fields}\label{S: orderable fields}
In this section we recall the main definitions and properties of totally ordered fields. For more details, we refer to Lang~\cite{lang} and van der Waerden~\cite{waerden}.
\begin{definition}[Orderable Field]\label{D: Orderable fields}
  A field $\mathbb K$ is \emph{orderable} if there exists a non-empty $\mathbb{K}_+ \subset\mathbb{K}$ such that 
  \begin{def-enumerate}
    \item $0\not\in\mathbb{K}_+$
    \item $(\forall x,y\in\mathbb{K}_+)(x + y\in\mathbb{K}_+ \emph{ and } xy\in\mathbb{K}_+)$ 
    \item $(\forall x\in\mathbb{K}\setminus\{0\})(x\in\mathbb{K}_+\emph{ or } -x\in\mathbb{K}_+)$
  \end{def-enumerate}
\end{definition}

  Provided that $\mathbb K$ is orderable, we can fix a set $\mathbb{K}_+$ that satisfies the properties given above and generate a strict order relation on $\mathbb K$ by $x<_{\mathbb{K}_+}y$ \ifff$y-x\in\mathbb{K}_+$. Further, we can define a total ordering (i.e. reflexive, anti-symmetric, transitive, and total) on $\mathbb K$ by $x\le_{\mathbb K_+}y$ \ifff $x<_{\mathbb K_+}y\emph{ or } x=y$. 

\begin{definition}[Totally Ordered Field]
  Let $\mathbb K$ be an orderable field and $\mathbb K_+\subset\mathbb K$. Then using the relation $\le_{\mathbb K_+}$ generated by $\mathbb K_+$, as noted above, we define $(\mathbb K,\le_{\mathbb K_+})$ to be a \emph{totally ordered field}. As well, with this ordering mind, we define the \emph{absolute value} of $x\in\mathbb{K}$ as $|x|=:\max(-x,x)$.
\end{definition}

For simplicity, when $\mathbb{K_+}$ is clear from the context we shall use the more standard notation $x\le y$ and $x<y$ in place of $x\le_{\mathbb{K}_+}y$ and $x<_{\mathbb{K}_+}y$, respectively. As well, we shall refer to $\mathbb K$ as being a totally ordered field -- or the less precise, ordered field -- rather than the more cumbersome $(\mathbb K,\le_{\mathbb K_+})$.


\begin{lemma}[Triangle Inequality]\label{L: 1 tri ineq}
  Let $\mathbb{K}$ be an ordered field and $a,b\in\mathbb{K}$, then $|a+b|\le|a|+|b|$.
\end{lemma}
\begin{proof}
  As $-|a|-|b|\le a+ b\le |a|+|b|$, we have $|a+b|\le|a|+|b|$ because $a+b\le|a|+|b|$ and $-(a+b)\le|a|+|b|$.
\end{proof}

\begin{definition}
  Let $\mathbb K$ be a totally ordered field and $A\subset \mathbb K$. We denote the set of upper bounds of $A$ by $\ub(A)$. More formally, $$\ub(A)=:\{x\in\mathbb K : (\forall a\in A)(a\le x)\}$$
\end{definition}
  
\begin{definition}[Ordered Field Homomorphism]
  Let $\mathbb K$ and $\mathbb F$ be ordered fields and $\phi:\mathbb K\to\mathbb F$ be a field homomorphism. Then $\phi$ is called an \emph{ordered field homomorphism} if it preserves the order; that is, $x\le y$ implies $\phi(x) \le\phi(y)$ for all $x, y\in\mathbb K$.
  Definitions of \emph{ordered field ismorphisms} and \emph{embeddings} follow similarly.
\end{definition}

\begin{remark}\label{R: q embedding}
  Let $\mathbb K$ be an ordered field. Then we define the ordered field embedding $\sigma:\mathbb Q\to\mathbb K$ by $\sigma(0)=:0$, $\sigma(n)=:n\cdot1$, $\sigma(-n)=:-\sigma(n)$ and $\sigma(\frac{m}{k})=:\frac{\sigma(m)}{\sigma(k)}$ for $n\in\mathbb N$ and $m,k\in\mathbb Z$. We say that $\sigma$ is the \emph{canonical embedding} of $\mathbb Q$ into $\mathbb K$.
\end{remark}

\begin{definition}[Formally Real]
  A field $\mathbb K$ is \emph{formally real} if, for every $n\in\mathbb N$, the equation $$\sum_{k=0}^n x_k^2 = 0$$ has only the trivial solution (that is, $x_k=0$ for each $k$) in $\mathbb K$.
\end{definition}

\begin{theorem}\label{T: orderable formally real}
  A field $\mathbb K$ is orderable \ifff $\mathbb K$ is formally real.
\end{theorem}
Further discussion, as well as the proof, of the preceding theorem can be found in van der Waerden~\cite{waerden} (chapter 11).

\begin{example}[Non-Orderable Field]
  \mbox{}
  \begin{enumerate}
    \item The field of complex numbers $\mathbb C$ is not orderable. Indeed, suppose there exists a subset $\mathbb C_+\subset \mathbb C$ that satisfies the properties above. Thus $i\in\mathbb C_+$ or $-i\in\mathbb C_+$. However, either case implies that $-1=(\pm i)^2\in\mathbb C_+$ and $1=(-1)(-1)\in\mathbb C_+$. Thus $0=1-1\in\mathbb C_+$, a contradiction. Therefore, $\mathbb C$ is non-orderable.
    \item The p-adic numbers $\mathbb Q_p$ are also non-orderable for similar reasons (see Ribenboim~\cite{riben} p.144-145 and Todorov \& Vernaeve~\cite{todor})
  \end{enumerate}
\end{example}

\begin{definition}[Real Closed Field]\label{D: real closed}
  Let $\mathbb K$ be a field. We say that $\mathbb K$ is a \emph{real closed field} if it satisfies the following.
  \begin{def-enumerate}
    \item $\mathbb K$ is formally real (or orderable).
    \item $(\forall a\in\mathbb K)(\exists x\in\mathbb K)(a=x^2\emph{ or }a=-x^2)$.
    \item $(\forall P\in\mathbb K[t])(\deg(P)\emph{ is odd}\Rightarrow(\exists x\in\mathbb K)(P(x)=0))$.
  \end{def-enumerate}
\end{definition}

\begin{theorem}\label{T: 1 order real-closed}
  Let $\mathbb K$ be a real closed totally ordered field and $x\in\mathbb K$. Then $x>0$ \ifff $x=y^2$ for some $y\in\mathbb K$. Thus every real closed field is ordered in a unique way.
\end{theorem}
\begin{proof}
  Suppose $x>0$, then there exists $y\in\mathbb K$ such that $x=y^2$ by part 2 of Definition~\ref{D: real closed}.

  Conversely, suppose $x=y^2$ for some $y\in\mathbb K$. By the definition of $\mathbb K_+$, we have $y^2\in\mathbb K_+$ for all $y\in\mathbb K$. Thus $x>0$.
\end{proof}

\begin{remark}
  If the field $\mathbb K$ is real closed, then we shall always assume that $\mathbb K$ is ordered by the unique ordering given above.
\end{remark}

\begin{lemma}\label{P: cont}
  Let $\mathbb K$ be an ordered field and $a\in\mathbb K$ be fixed. The scaled identity function $a\cdot id(x)=:ax$ is uniformly continuous in the order topology on $\mathbb K$. Consequently, every polynomial in $\mathbb K$ is continuous.
\end{lemma}
\begin{proof}
  Given $\epsilon\in\mathbb K_+$, let $\delta = \frac{\epsilon}{|a|}$. Indeed, $(\forall x,y\in\mathbb K)(|x-y|<\delta\Rightarrow |ax-ay|=|a||x-y|<|a|\delta=\epsilon)$.
\end{proof}

\begin{lemma}[Intermediate Value Theorem]\label{L: ivt}
  Let $\mathbb K$ be an ordered field. As well, let $f:H\to\mathbb K$ be a function that is continuous in the order topology on $\mathbb K$ and $[a,b]\subset H$. If $\mathbb K$ is Dedekind complete (in the sense of sup), then, for any $u\in\mathbb K$ such that $f(a)\le u\le f(b)$ or $f(b)\le u\le f(a)$, there exists a $c\in [a,b]$ such that $f(c)=u$.
\end{lemma}
\begin{proof}
  We will only show the case when $f(a)\le f(b)$, the other should follow similarly.

  Let $S=:\{x\in [a,b]:f(x)\le u\}$. We observe that $S$ is non-empty as $a\in S$ and that $S$ is bounded above by $b$; thus, $c=:\sup(S)$ exists by assumption. As well, we observe that $c\in[a,b]$ because $c\le b$. To show that $f(c)=u$ we first observe that, as $f$ is continuous, we can find $\delta\in\mathbb K_+$ such that $(\forall x\in\mathbb K)(|x-c|<\delta \Rightarrow |f(c)-f(x)|<|f(c)-u|)$. 

  If $f(c)>u$, then, from our observation, it follows that $f(x)>f(c)-(f(c)-u)=u$ for all $x\in(c-\delta, c+\delta)$. Thus $c-\delta\in\ub(S)$, which contradicts the minimality of $c$. 

  Similarly, if $u>f(c)$, then, from our observation it follows that $u=f(c)+(u-f(c)>f(x)$ for all $x\in(c-\delta, c+\delta)$, which contradicts $c$ being an upper bound.

  Therefore $f(c)=u$ as $\mathbb K$ is totally ordered.
\end{proof}

\begin{remark}
  When dealing with polynomials, it follows from the Artin-Schrier Theorem that Dedekind completeness is not necessary to produce the results of the Intermediate Value Theorem. For a general reference, see Lang~\cite{lang}, Chapter XI.
\end{remark}

\begin{theorem}\label{T: order complete -> real closed}
  Let $\mathbb K$ be a totally ordered field which is also Dedekind complete. Then $\mathbb K$ is real closed.
\end{theorem}
\begin{proof}
  First observe that $\mathbb K$ is formally real because it is orderable.

  Now let $a\in\mathbb K_+$ and $S=:\{x\in\mathbb K: x^2<a\}$. Observe that $0\in S$ and that $m=:\max\{1,a\}$ is an upper bound of $S$. Indeed, when $a\le 1$, we have $x^2<1$, which implies $x<1$ for all $x\in S$. On the other hand, when $1<a$, we have $x^2<a<a^2$; thus, $x<a$ for all $x\in S$. From this observation, it follows that $s=:\sup S$ exists. We intend to show that $s^2=a$.

      \begin{description}
          \item[Case $(s^2<a)$:] Let $h=:\frac{1}{2}\min\{\frac{a-s^2}{(s+1)^2}, 1\}$. From this definition, it follows that
            \begin{equation}
              2h\le \frac{a-s^2}{(s+1)^2}\emph{\;\;and\;\;}2h\le 1\label{Eq: h obs}
            \end{equation}
            We wish to show that $(s+h)^2<a$.
       
            From $0<h\le\frac{1}{2}$ we have $h<s^2+1$, which implies that $h+2s<(s+1)^2$ and $h(h+2s) + s^2 < h(s+1)^2 + s^2$. Thus, we have $(s+h)^2<h(s+1)^2 + s^2$. By (\ref{Eq: h obs}), we know that $h(s+1)^2<2h(s+1)^2\le a-s^2$. Thus, we have $(s+h)^2<h(s+1)^2 + s^2 < a$. Therefore $(s+h)\in S$ which contradicts $s$ being an upper bound.
       
          \item[Case $(s^2>a)$:] Let $h=:\frac{s^2-a}{2(s+1)^2}$. First we observe that, from the definition of $h$, $s^2-a=2h(s+1)^2>h(s+1)^2$. We intend to show that $(s-h)^2>a$. Indeed, we obviously have $s^2+1>-h$ which implies that $(s+1)^2>2s-h$ and $h(s+1)^2>h(2s-h)$. Thus $s^2-h(s+1)^2<s^2-h(2s-h)=(s-h)^2$ and by our observation, we find that $a<s^2-h(s+1)^2<(s-h)^2$. Therefore $(s-h)$ is an upper bound of $S$, which contradicts the minimality of $s$.
        \end{description}

  Finally, to show that every odd degree polynomial $P(x)\in\mathbb K[x]$ has a root, we observe that $lim_{x\to-\infty}P(x)=-\lim_{x\to\infty}P(x)$. Combining this result with Lemma~\ref{P: cont} and Lemma~\ref{L: ivt}, we find that $\exists c\in\mathbb K$ such that $P(c)=0$.
\end{proof}

\section{Infinitesimals in Ordered Fields}\label{S: inf}
In this section we recall the definitions of infinitely small (infinitesimal), finite and infinitely large elements in totally ordered fields and study their basic properties. As well, we present a characterization of Archimedean fields in the languague of infinitesimals and infinitely large elements.

\begin{definition}[Archimedean Property]\label{D: Archimedean Field}
  A totally ordered field (ring) $\mathbb{K}$ is \emph{Archimedean} if for every $x\in\mathbb{K}$, there exists $n\in\mathbb{N}$ such that $|x|<n$. If $\mathbb K$ is Archimedean, we also may refer to $\mathbb K(i)$ as Archimedean. If $\mathbb K$ is not Archimedean, then we refer to $\mathbb K$ as \emph{non-Archimedean}.
\end{definition}

For the rest of the section we discuss the properties of Archimedean and non-Archimedean fields through the characteristics of infinitesimals.

\begin{definition}\label{D: infinitesimal}
  Let $\mathbb{K}$ be a totally ordered field. We define
  \begin{def-enumerate}
    \item $\mathcal{I}(\mathbb{K})=:\{x\in\mathbb{K} : (\forall n\in\mathbb{N})(|x|<\frac{1}{n})\}$
    \item $\mathcal{F}(\mathbb{K})=:\{x \in\mathbb{K} : (\exists n\in \mathbb{N})(|x|\le n)\}$
    \item $\mathcal{L}(\mathbb{K})=:\{x \in\mathbb{K} : (\forall n\in \mathbb{N})(n<|x|)\}$
  \end{def-enumerate}
  The elements in $\mathcal I(\mathbb K), \mathcal F(\mathbb K), \textrm{ and } \mathcal L(\mathbb K)$ are referred to as \emph{infinitesimal (infinitely small), finite and infinitely large}, respectively. We sometimes write $x\approx0$ if $x\in\mathcal I(\mathbb K)$ and $x\approx y$ if $x-y\approx 0$, in which case we say that $x$ is \emph{infinitesimally close} to $y$.
\end{definition}

\begin{proposition} For a totally ordered field $\mathbb K$, we have the following properties for the sets given above.
  \begin{thm-enumerate}
    \item $\mathcal I(\mathbb K)\subset \mathcal F(\mathbb K)$.
    \item $\mathbb K=\mathcal F(\mathbb K)\cup\mathcal L(\mathbb K)$.
    \item $\mathcal F(\mathbb K)\cap\mathcal L(\mathbb K)=\emptyset$.
    \item If $x\in \mathbb K\setminus \{0\}$ then $x\in\mathcal I(\mathbb K)$ iff $\frac{1}{x}\in\mathcal L(\mathbb K)$.
  \end{thm-enumerate}
\end{proposition}
\begin{proof}
  \begin{thm-enumerate}
    \item Let $\alpha\in\mathcal I(\mathbb K)$, then $\alpha<1$, therefore $\alpha\in\mathcal F(\mathbb K)$.
    \item Suppose $x\in\mathbb K$, then either $(\exists n\in\mathbb N)(|x|<n)$ or $(\forall n\in\mathbb N)(n\le|x|)$. Thus $x\in\mathcal F(\mathbb K)$ or $x\in\mathcal L(\mathbb K)$. The other direction follows from the definition.
    \item Suppose $x\in\mathcal F(\mathbb K)\cap\mathcal L(\mathbb K)$, then $\exists n\in\mathbb N$ such that $|x|<n$, but, we also have $(\forall m\in\mathbb N)(m<|x|)$; thus $n<|x|<n$, a contradiction.
    \item Suppose $x\in\mathbb K\setminus\{0\}$. Then $x\in\mathcal L(\mathbb K)$ iff $(\forall n\in\mathbb N)(n<|x|)$ iff $(\forall n\in\mathbb N)(|\frac{1}{x}|<\frac{1}{n})$ iff $\frac{1}{x}\in\mathcal I(\mathbb K)$
  \end{thm-enumerate}
\end{proof}

\begin{proposition}[Characterizations]\label{P: 1 arch}
  Let $\mathbb{K}$ be a totally ordered field. Then the following are equivalent:
  \begin{thm-enumerate}
    \item $\mathbb{K}$ is Archimedean.
    \item $\mathcal L(\mathbb K)=\emptyset$.
    \item $\mathcal I(\mathbb K)=\{0\}$.
    \item $\mathcal{F}(\mathbb{K})=\mathbb{K}$.
  \end{thm-enumerate}
\end{proposition}
\begin{proof}
  \begin{description}
    \item[$(i)\Rightarrow(ii)$] Follows from the definition of Archimedean field.
    \item[$(ii)\Rightarrow(iii)$] Suppose $d\alpha\in\mathcal{I}(\mathbb{K})$ such that $d\alpha\neq0$. As $\mathbb{K}$ is a field, $d\alpha^{-1}$ exists. Thus $d\alpha<\frac{1}{n}$, for all $n\in\mathbb{N}$, which gives $1<\frac{1}{n}d\alpha^{-1}$ for all $n\in \mathbb{N}$. Therefore $n<d\alpha^{-1}$ for all $n\in\mathbb{N}$, which means $d\alpha\in \mathcal{L}(\mathbb{K})$.
    \item[$(iii)\Rightarrow(iv)$] Note that we clearly have $\mathcal{F}(\mathbb{K})\subseteq\mathbb{K}$. Suppose, to the contrary, there exists $\alpha\in\mathbb{K}\setminus\mathcal{F}(\mathbb{K})$. Then, by definition, $|\alpha|>n$ for all $n\in\mathbb{N}$; hence $\frac{1}{n}>\frac{1}{|\alpha|}$ for all $n\in\mathbb{N}$ because $\mathbb{K}$ is a field. Thus $\frac{1}{|\alpha|}\in\mathcal{I}(\mathbb{K})$ so that $\frac{1}{|\alpha|}=0$, a contradiction.
    \item[$(iv)\Rightarrow(i)$] By definition of $\mathcal{F}(\mathbb{K})$, we know that for every $\alpha\in\mathbb{K}=\mathcal{F}(\mathbb{K})$ there exists a $n\in\mathbb{N}$ such that $|\alpha|<n$; hence $\mathbb{K}$ is Archimedean.
  \end{description}
\end{proof}

\begin{lemma}\label{L: finite arch ring}
  Let $\mathbb{K}$ be a totally ordered field. Then
  \begin{thm-enumerate}
  \item  $\mathcal{F}(\mathbb{K})$ is an Archimedean ring.
  \item  $\mathcal{I}(\mathbb{K})$ is a maximal ideal of $\mathcal{F}(\mathbb{K})$. Moreover, $\mathcal{I}(\mathbb{K})$ is a \emph{convex ideal} in the sense that $a\in\mathcal{F}(\mathbb{K})$ and $|a|\le|b|\in\mathcal{I}(\mathbb{K})$ implies $a\in\mathcal{I}(\mathbb{K})$.
  \end{thm-enumerate}
  Consequently $\mathcal{F}(\mathbb{K})/\mathcal{I}(\mathbb{K})$ is a totally ordered Archimedean field.
\end{lemma}     
\begin{proof}   
  \mbox{}
  \begin{thm-enumerate}
  \item  The fact that $\mathcal{F}(\mathbb{K})$ is Archimedean follows directly from its definition.  Observe that $|-1|\le1$, therefore $-1\in\mathcal{F}(\mathbb{K})$. Suppose that $a,b,c\in\mathcal{F}(\mathbb{K})$, then $|a|\le n$, $|b|\le m$ and $|c|\le k$ for some $n,m,k\in\mathbb{N}$. Thus $|ab+c|\le|a||b| + |c|\le nm + k\in\mathbb{N}$ by Lemma~\ref{L: 1 tri ineq}, which implies $ab+c\in\mathcal{F}(\mathbb{K})$.
  \item  Let $x,y\in\mathcal{I}(\mathbb{K})$. Then, for any $n\in\mathbb{N}$, we have $|x +y|\le|x|+|y|<\frac{1}{2n}+\frac{1}{2n}=\frac{1}{n}$; thus $x+y\in\mathcal{I}(\mathbb{K})$.

  Now suppose $a\in\mathcal{I}(\mathbb{K})$ and $b\in\mathcal{F}(\mathbb{K})$. Then $|b|\le n$ for some $n\in\mathbb{N}$. As $|a|<\frac{1}{nm}$ for all $m\in\mathbb{N}$, we have $|ab|\le\frac{n}{nm}=\frac{1}{m}$ for all $m\in\mathbb{N}$. Hence, $ab\in\mathcal{I}(\mathbb{K})$.

  Suppose there exists an ideal $R\subseteq\mathcal{F}(\mathbb{K})$ that properly contains $\mathcal{I}(\mathbb{K})$ and let $k\in R\setminus\mathcal{I}(\mathbb{K})$. Then $\frac{1}{n}\le |k|$ for some $n\in\mathbb{N}$, hence $n\ge\frac{1}{|k|}\in\mathbb{K}$ which implies $\frac{1}{k}\in\mathcal{F}(\mathbb{K})$ and $1=\frac{k}{k}\in R$. Therefore $R=\mathcal{F}(\mathbb{K})$.
  
  Finally, let $b\in\mathcal{I}(\mathbb{K})$. Suppose $a\in\mathcal{F}(\mathbb{K})$ such that $|a|<|b|$. Then $|a|<|b|< \frac{1}{n}$ for all $n\in\mathbb{N}$. Therefore $a\in\mathcal{I}(\mathbb{K})$.
  \end{thm-enumerate}
\end{proof}     
  
\begin{remark}
  Archimedean rings (which are not fields) might have non-zero infinitesimals. For example, $\mathcal F(\mathbb K)$ is always an Archimedean ring, but it has non-zero infinitesimals when $\mathbb K$ is a non-Archimedean field.
\end{remark}
  
\begin{example}[Archimedean Fields]
  The fields $\mathbb R, \mathbb Q, \mathbb C$ are all Archimedean fields.
\end{example}
  
For examples of non-Archimedean fields, we refer the reader to \S~\ref{S: examples 1} and \S~\ref{S: examples 2}.
  
  
\section{Completeness of an Archimedean Field}\label{S: CAF}
  
In what follows, convergence is meant in reference to the order topology on $\mathbb K$. As well, the reader should recall that there is a natural embedding of the rationals (and, thus the natural numbers) into any totally ordered field (see remark~\ref{R: q embedding}).
  
In what follows, we provide several definitions of rather well known forms of completeness that will be used throughout the rest of this section.

\begin{definition}[Completeness]\label{D: completeness}
  Let $\mathbb{K}$ be a totally ordered field.
  \begin{def-enumerate}
  \item Let $\kappa$ be an uncountable cardinal. Then $\mathbb{K}$ is \emph{Cantor $\kappa$-complete} if every family $\{[a_\gamma,b_\gamma]\}_{\gamma\in \Gamma}$ of fewer than $\kappa$ closed bounded intervals in $\mathbb{K}$ with the finite intersection property (F.I.P.) has a non-empty intersection, $\bigcap_{\gamma\in \Gamma} [a_\gamma,b_\gamma]\neq \emptyset$. If $\mathbb{K}$ is Cantor $\aleph_1$-complete, where $\aleph_1=\aleph_0^+$ (the successor of $\aleph_0=\card(\mathbb{N})$), then we say that $\mathbb{K}$ is \emph{Cantor complete}. The latter means that every nested sequence of bounded closed intervals in $\mathbb{K}$ has a non-empty intersection.
  \item Let $^*\mathbb K$ be a non-standard extension of $\mathbb K$  (see either Lindstr\o m~\cite{lindstrom} or Davis~\cite{davis}) and let $\mathcal F(^*\mathbb K)$ and $\mathcal I(^*\mathbb K)$ be the sets of finite and infinitesimal elements in $^*\mathbb K$, respectively (see Definition~\ref{D: infinitesimal}). Then we say that $\mathbb{K}$ is \emph{Leibniz complete} if for every $\alpha\in\mathcal{F}(^*\mathbb{K})$, there exists unique $L\in\mathbb K$ and $dx\in\mathcal{I}(^*\mathbb{K})$ such that $\alpha=L+dx$; we will sometimes denote this by $\mathcal{F}(^*\mathbb{K})=\mathbb{K}\oplus \mathcal{I}(^*\mathbb{K})$ which is equivalent to saying $\mathcal{F}(^*\mathbb{K})/\mathcal{I}(^*\mathbb{K})=\mathbb{K}$.
  \item $\mathbb{K}$ is \emph{Dedekind complete} if every non-empty subset of $\mathbb{K}$ that is bounded from above has a supremum.
  \item $\mathbb{K}$ is \emph{sequentially complete} if every fundamental (Cauchy) sequence in $\mathbb{K}$ converges. Recall that a sequence $\{a_n\}$ in a totally ordered field $\mathbb{K}$ (not necessarily Archimedean) is called \emph{fundamental} if for all $\epsilon\in\mathbb{K}_+$, there exists an $N\in\mathbb{N}$ such that for all $n,m\in\mathbb{N}$, $n,m\ge N$ implies that $|a_n-a_m|<\epsilon$.
  \item We say that $\mathbb{K}$ is \emph{Bolzano-Weierstrass complete} if every bounded sequence has a convergent subsequence.
  \item We say that $\mathbb{K}$ is \emph{Bolzano complete} if every bounded infinite set has a cluster point.
  \item We say that $\mathbb{K}$ is \emph{monotone complete} if every bounded monotonic sequence is convergent.
  \item Suppose that $\mathbb{K}$ is Archimedean. Then $\mathbb{K}$ is \emph{Hilbert complete} if $\mathbb{K}$ has no proper totally ordered Archimedean field extensions.
  \end{def-enumerate}
\end{definition}

\begin{remark}[Completeness of the Reals in History]
  \mbox{}
  \begin{def-enumerate}
  \item Leibniz completeness (number 2 in Definition~\ref{D: completeness} above) appears in the early Leibniz-Euler Infinitesimal Calculus as the statement that ``every finite number is infinitesmially close to a unique usual quantity.'' Here the ``usual quantities'' are what we now refer to as the real numbers and can be identified with $\mathbb{K}$ in the definition above. This form of completeness was more or less, always treated as an obvious fact; what was not obvious, and a possible reason for the demise of the infinitesimals, was the validity of what has come to be known as the Leibniz Principle (see H. J. Keisler~\cite{keisler} p. 42 and Stroyan \& Luxemburg~\cite{strolux76} p. 22), which roughly states that there is a non-Archimedean field extension $^*\mathbb{K}$ of $\mathbb{K}$ such that every function $f$ on $\mathbb{K}$ has an extension $^*f$ to $^*\mathbb{K}$ that ``preserves all the properties of $\mathbb{K}$.'' For example, $(x+y)^2=x^2+2xy+y^2$ and $^*\sin(x+y)={^*\sin(x)}{^*\cos(y)}+{^*\sin(y)}{^*\cos(x)}$ hold in $^*\mathbb{K}$ because the analogous statements hold in $\mathbb{K}$ (note that $^*\sin$ is usually written as $\sin$ for convenience). All attempts to construct a field with such properties failed until the 1960's when A. Robinson developed the theory of non-standard analysis along with the Transfer Principle, which is analogous to the Leibniz Principle, and proved that every field $\mathbb{K}$ has a non-standard extension $^*\mathbb{K}$. For a detailed exposition, we refer to Lindstr\o m~\cite{lindstrom}, Davis~\cite{davis} and Keisler~\cite{keisler},\cite{keisler2}.
  \item Dedekind completeness was introduce by Dedekind (independently from many others, see O'Connor~\cite{oconnor}) at the end of the 19th century. From the point of view of modern mathematics, Dedekind proved the consistency of the axioms of the real numbers by constructing an example from Dedekind cuts.
  \item Sequential completeness, listed as number 4 above, is a well known form of completeness of metric spaces, but it has also been used in constructions of the real numbers: Cantor's construction using Cauchy sequences (see O'Connor~\cite{oconnor}), an example of which can be found in Hewitt \& Stromberg~\cite{hewitt}.
  \item Cantor completeness (also known as the ``nested interval property''), monotone completeness, Bolzano-Weierstrass completeness, and Bolzano completeness typically appear in real analysis as ``theorems'' or ``important principles'' rather than as forms of completeness; however, in non-standard analysis, Cantor completeness takes a much more important role along with the concept of algebraic saturation which is defined in Definition~\ref{D: algebraic saturation}.
  \item Hilbert completeness, listed as number 8 above, is a less well-known form of completeness that was originally introduced by Hilbert in 1900 with his axiomatic definition of the real numbers (see Hilbert~\cite{hilbert} and O'Connor~\cite{oconnor}).
  \end{def-enumerate}
\end{remark}

\begin{theorem}\label{T: exist ded}
  There exists a Dedekind complete field.
\end{theorem}
\begin{proof}
  For the proof of this we refer to either the construction by means of Dedekind cuts in Rudin~\cite{poma} or the construction using equivalence classes of Cauchy sequences in Hewitt \& Stromberg~\cite{hewitt}, or to \S~\ref{S: construct} of this text where we present a construction using non-standard analysis.
\end{proof}

\begin{lemma}\label{L: ded cuts properties}
  Let $\mathbb{A}$ be an Archimedean field and $\mathbb{K}$ be a Dedekind complete field. Define $C_\alpha=:\{q\in\mathbb{Q}: q<\alpha\}$ for $\alpha\in\mathbb{A}$. Then
  \begin{thm-enumerate}
  \item For all $\alpha,\beta\in\mathbb{A}$ we have $\sup_\mathbb{K}(C_{\alpha+\beta})=\sup_\mathbb{K}(C_\alpha)+\sup_\mathbb{K}(C_\beta)$.
  \item For all $\alpha,\beta\in\mathbb{A}$ we have  $\sup_\mathbb{K}(C_{\alpha\beta})=\sup_\mathbb{K}(C_\alpha)\sup_\mathbb{K}(C_\beta)$.
  \end{thm-enumerate}
\end{lemma}
\begin{proof}
  First note that, as both $\mathbb{A}$ and a $\mathbb{K}$ are ordered fields, they each contain a copy of $\mathbb{Q}$ so that our claim actually makes sense.
  \begin{thm-enumerate}
  \item We only need to show that $C_{\alpha+\beta}=C_\alpha+C_\beta$, where $C_\alpha+C_\beta=:\{q+p:q\in C_\alpha,p\in C_\beta\}$, and the rest will follow from properties of supremum. It should be clear that $C_\alpha+C_\beta\subseteq C_{\alpha+\beta}$, thus suppose $p\in C_{\alpha+\beta}$. Let $r\in\mathbb{Q}_+$ and $u\in\mathbb{Q}$ be such that $0<r<\alpha+\beta - p$ and $\alpha-r<u<\alpha$ (which is possible because the rationals are dense in any Archimedean field). Then $p-u<p+r-\alpha<\beta$, and as $p,u\in\mathbb{Q}$, we know that $p-u\in\mathbb{Q}$ so if we define $v=:p-u$, then $u\in C_\alpha, v\in C_\beta$ and $u+v=p$.
  \item To show this, we will first prove the case when both $\alpha$ and $\beta$ are positive. Let $P_\gamma=:C_\gamma\cap \mathbb{Q}_+$ for $\gamma\in\mathbb{A}$ and suppose both $\alpha$ and $\beta$ are positive; then $\sup_\mathbb{K}(C_\alpha)=\sup_\mathbb{K}(P_\alpha)$ and $\sup_\mathbb{K}(C_\beta)=\sup_\mathbb{K}(P_\beta)$. As before, we will show that $P_{\alpha\beta}=P_\alpha\cdot P_\beta=:\{qp : q\in P_\alpha,p\in P_\beta\}$ from which the desired result will follow by basic properties of supremum. Note that we clearly have $P_\alpha\cdot P_\beta\subseteq P_{\alpha\beta}$. Suppose $p\in P_{\alpha\beta}$, let $u\in\mathbb{Q}$ be such that $\frac{p}{\beta}<u<\alpha$ and define $v=:\frac{p}{u}$. Then $v<\frac{p}{p/\beta}=\beta$ so that $u\in P_\alpha, v\in P_\beta$ and $uv=p$; thus $p\in P_\alpha\cdot P_\beta$. Therefore $\sup_\mathbb{K}(C_\alpha)\sup_\mathbb{K} (C_\beta)=\sup_\mathbb{K}(C_{\alpha\beta})$

    For the remaining cases, first note that if $q<\alpha$, then $-\alpha<-q$ so that $\sup_\mathbb{K}(C_{-\alpha})<-q$, thus $q<-\sup_\mathbb{K}(C_{-\alpha})$ which implies that $\sup_\mathbb{K}(C_{\alpha})\le -\sup_\mathbb{K}(C_{-\alpha})$; hence, by symmetry, $\sup_\mathbb{K}(C_{-\alpha})=-\sup_\mathbb{K}(C_\alpha)$. So, if either $\alpha$ or $\beta$ are zero, then $\sup_\mathbb{K}(C_0)=-\sup_\mathbb{K}(C_0)$ so that $\sup_\mathbb{K}(C_0)=0$ and thus our desired result holds. If instead, both $\alpha$ and $\beta$ are negative, then we have $\sup_\mathbb{K}(C_\alpha)\sup_\mathbb{K}(C_\beta)= \sup_\mathbb{K}(C_{-\alpha})\sup_\mathbb{K}(C_{-\beta}) =\sup_\mathbb{K}(C_{(-\alpha)(-\beta)}) =\sup_\mathbb{K}(C_{\alpha\beta})$. Otherwise, if $\alpha\beta<0$, then without loss of generality, we can assume that $\alpha>0$ and $\beta<0$ so that we have $\sup_\mathbb{K}(C_\alpha)\sup_\mathbb{K}(C_\beta)=-\sup_\mathbb{K}(C_\alpha)\sup_\mathbb{K}(C_{-\beta})=-\sup_\mathbb{K}(C_{-\alpha\beta})=\sup_\mathbb{K}(C_{\alpha\beta})$.
  \end{thm-enumerate}
\end{proof}

\begin{theorem}\label{T: arch embed ded}
  Let $\mathbb{A}$ be a totally ordered Archimedean field and let $\mathbb{K}$ be a totally ordered Dedikind complete field. Then the mapping $\sigma:\mathbb{A}\to\mathbb{K}$ given by $\sigma(\alpha)=:\sup_\mathbb{K}(C_\alpha)$, where $C_\alpha=:\{q\in\mathbb{Q}: q<\alpha\}$, is an order field embedding of $\mathbb{A}$ into $\mathbb{K}$.
\end{theorem}
\begin{proof}
  Recall that $\mathbb{Q}$ can be embedded into both $\mathbb{A}$ and $\mathbb{K}$ as they are ordered fields. Note that for $\alpha\in\mathbb{A}$, there exists $q,p\in\mathbb{Q}$ such that $q<\alpha<p$ by the Archimedean property; thus the set $C_\alpha$ is both non-empty and bounded from above in $\mathbb{A}$ and $\mathbb{K}$. Now let $\alpha,\beta\in\mathbb{A}$. To show that $\sigma$ preserves order, suppose $\alpha<\beta$; then $C_\alpha\subseteq C_\beta$, and because $\mathbb{Q}$ is dense in every Archimedean field, $C_\alpha\neq C_\beta$, therefore $\sigma(\alpha)<\sigma(\beta)$. The fact that $\sigma$ is an isomorphism follows from Lemma~\ref{L: ded cuts properties} above.
\end{proof}

\begin{corollary}\label{C: ded order iso}
  All Dedekind complete fields are mutually order-isomorphic. Consequently they have the same cardinality, which is usually denoted by $\mathfrak c$.
\end{corollary}
\begin{proof}
  Let $\mathbb{K}$ and $\mathbb{F}$ be Dedekind complete fields. Using the mapping $\sigma:\mathbb{K}\to\mathbb{F}$ from the preceding proof, for any $k\in \mathbb{F}$ it is easy to see that $\sup_\mathbb{K} C_k$ maps to $k$.
\end{proof}

\begin{corollary}\label{C: arch card}
  Every Archimedean field has cardinality at most $\mathfrak c$.
\end{corollary}

The next result shows that an Archimedean field can never be order-isomorphic to one of its proper subfields.

\begin{theorem}
  Let $\mathbb{K}$ be a totally ordered Archimedean field and $\mathbb{F}$ be a subfield of $\mathbb{K}$. If $\sigma:\mathbb{K}\to\mathbb{F}$ is an order-isomorphism between $\mathbb{K}$ and $\mathbb{F}$, then $\mathbb{K}=\mathbb{F}$ and $\sigma=id_\mathbb{K}$.
\end{theorem}
\begin{proof}
  Suppose $\sigma:\mathbb{K}\to\mathbb{F}$ is an order-preserving isomorphism. Note that, as an isomorphism, $\sigma$ fixes the rationals. Let $a\in\mathbb{K}$ and $A=:\{q\in\mathbb{Q}:q<a\}$. Recall that the rationals are dense in an Archimedean field, hence we have $a=\sup_\mathbb{K}A$. Then because $\sigma$ is order preserving and fixes $\mathbb{Q}$, we know that $\sigma(a)\in\ub(A)$ and thus $\sigma(a)\ge a$. To show that $\sigma(a)=a$, suppose, to the contrary that $\sigma(a)>a$. Then we can find a rational $q$ such that $a<q<\sigma(a)$, but $\sigma$ is order-preserving, so $\sigma(a)<\sigma(q)=q$, a contradiction. Therefore $\sigma=id_\mathbb{K}$ and $\mathbb{K}=\mathbb{F}$.
\end{proof}

As a counter-example to the preceding theorem for when $\mathbb{K}$ is non-Archimedean we have the field of rational functions $\mathbb{R}(x)$.

\begin{example}
  Let $\mathbb{R}(x)$ be the field of rational functions over $\mathbb{R}$ with indeterminate $x$ and supply the field with an ordering given by $f< g$ if and only if there exists a $N\in\mathbb{N}$ such that $g(x)-f(x)>0$ for all $x\in\mathbb{R}, x\ge N$. Then the field $\mathbb{R}(x^2)$ is a proper subfield of $\mathbb{R}(x)$ which is order-isomorphic to $\mathbb{R}(x)$ under the map $\sigma:\mathbb{R}(x)\to \mathbb{R}(x^2)$ given by $\sigma(f(x))=f(x^2)$ for all $f(x)\in\mathbb{R}(x)$.
\end{example}

\begin{lemma}\label{L: order -> arch}
  Let $\mathbb{K}$ be a Dedekind complete ordered field, then $\mathbb{K}$ is Archimedean.
\end{lemma}
\begin{proof}
  Suppose, to the contrary, that $\mathbb{K}$ is non-Archimedean, then $\mathbb{N}\subset\mathbb{K}$ is bounded from above. Let $\alpha\in\mathbb{K}$ be the least upper bound of $\mathbb{N}$. Then it must be that, for some $n\in\mathbb{N}$, $\alpha-1< n$, thus $\alpha< n+1$, a contradiction.
\end{proof}
                
The following theorem, shows that, under the assumption that we are working with a totally ordered Archimedean field, all of the forms of completeness listed in Definition~\ref{D: completeness} are in fact equivalent. Later (in \S~\ref{S: non-arch completeness}) we will examine these properties without the assumption that the field is Archimedean, to find that this equivalence does not necessarily hold.

\begin{theorem}[Completeness of an Archimedean Field]\label{T: BIG ONE} Let $\mathbb{K}$ be a totally ordered Archimedean field. Then the following are equivalent. 
  \begin{thm-enumerate}
    \item $\mathbb{K}$ is Cantor $\kappa$-complete for any infinite cardinal $\kappa$.
     
    \item $\mathbb{K}$ is Leibniz complete (see remark~\ref{R: NSA Unique} for uniqueness).
                     
    \item $\mathbb{K}$ is monotone complete.
                     
    \item $\mathbb{K}$ is Cantor complete (i.e. Cantor $\aleph_1$-complete, not for all cardinals).
                     
    \item $\mathbb{K}$ is Bolzano-Weierstrass complete.
      
    \item $\mathbb{K}$ is Bolzano complete.
                     
    \item $\mathbb{K}$ is sequentially complete.
                     
    \item $\mathbb{K}$ is Dedekind complete.

    \item $\mathbb{K}$ is Hilbert complete.
  \end{thm-enumerate}
\end{theorem}
 
\begin{proof}
  \mbox{} 
  \begin{description}
  \item[$(i)\Rightarrow(ii)$:] Let $\kappa$ be the successor of $\card(\mathbb{K})$. As well, let $\alpha\in\mathcal F(^*\mathbb K)$ and $S=:\{[a,b]:a,b\in\mathbb K\emph{ and } a\le \alpha\le b\textrm{ in }{^*\mathbb K}\}$. Clearly $S$ satisfies the finite intersection property and $\card(S)=\card(\mathbb{K}\times\mathbb{K})=\card(\mathbb{K})<\kappa$; thus, by assumption, there exists $L\in \bigcap_{[a,b]\in S} [a,b]$. To show $\alpha-L\in\mathcal I(^*\mathbb K)$, suppose, to the contrary, that $\alpha-L\not\in\mathcal I(^*\mathbb K)$, i.e. $\frac{1}{n}<|\alpha-L|$ for some $n\in\mathbb N$. Then either $\alpha<L-\frac{1}{n}$ or $L+\frac{1}{n}<\alpha$. However the former implies $L\le L-\frac{1}{n}$ and the latter implies $L+\frac{1}{n}\le L$. In either case we reach a contradiction, therefore $\alpha-L\in\mathcal I(^*\mathbb K)$.
 
  \item[$(ii)\Rightarrow(iii)$:] Let $\{x_n\}_{n\in\mathbb N}$ be a bounded monotonic sequence in $\mathbb K$; without loss of generality, we can assume that $\{x_n\}$ is increasing. We denote the non-standard extension of $\{x_n\}_{n\in\mathbb N}$ by $\{^*x_{\nu}\}_{\nu \in {^*\mathbb N}}$. Observe that, by the Transfer Principle (see Davis~\cite{davis}), $\{^*x_{\nu}\}$ is increasing as $\{x_n\}$ is increasing. Also (by the Transfer Principle), for any $b\in\ub(\{x_n\})$ we have $(\forall \nu\in{^*\mathbb N})(^*x_{\nu}\le b)$. Now choose $\nu \in\mathcal L(^*\mathbb N)$. Then $^*x_{\nu}\in\mathcal F(^*\mathbb K)$, because $\{^*x_{\nu}\}$ is bounded by a standard number; thus, there exists $L\in\mathbb K$ such that $L\approx {^*x_{\nu}}$ by assumption. Since $\{^*x_{\nu}\}$ is increasing, it follows that $L\in\ub(\{x_n\})$. To show that $x_n\to L$, suppose that it does not. Then, there exists $\epsilon\in\mathbb K_+$ such that $(\forall n\in\mathbb N)(L-x_n\ge\epsilon)$. Thus, we have $L-\epsilon\in\ub(\{x_n\})$, which implies $^*x_{\nu}\le L-\epsilon$ (by the transfer principle) contradicting $^*x_{\nu}\approx L$.
 
  \item[$(iii)\Rightarrow(iv)$:] Suppose that $\{[a_i,b_i]\}_{i\in\mathbb{N}}$ satisfies the finite intersection property. Let $\Gamma_n=:\cap_{i=1}^n[a_i,b_i]$ and observe that $\Gamma_n=[\alpha_n, \beta_n]$ where $\alpha_n=:\max_{i\le n} a_i$ and $\beta_n=:\min_{i\le n}b_i$. Then $\{\alpha_n\}_{n\in\mathbb N}$ is a bounded increasing sequence and $\{\beta_n\}_{n\in\mathbb N}$ is a bounded decreasing sequence; thus $\alpha=:\lim_{n\to\infty}\alpha_n$ and $\beta=:\lim_{n\to\infty}\beta_n$ exist by assumption. If $\beta<\alpha$, then for some $n$ we would have $\beta_n<\alpha_n$, a contradiction; hence, $\alpha\le\beta$. Therefore $\cap_{i=1}^{\infty}[a_i,b_i]=[\alpha,\beta]\neq\emptyset$.
 
  \item[$(iv)\Rightarrow(v)$:] This is the familiar \textbf{Bolzano-Weierstrass Theorem} (Bartle \& Sherbert~\cite{bartle}, p. 79). Let $\{x_n\}_{n\in\mathbb{N}}$ be a bounded sequence in $\mathbb{K}$, then there exists $a,b\in\mathbb K$ such that $\{x_n:n\in\mathbb N\}\subset [a,b]$. Let $\Gamma_1=:[a,b]$, $n_1=:1$ and divide $\Gamma_1$ into two equal subintervals $\Gamma_1'$ and $\Gamma_1''$. Let $A_1=:\{n\in\mathbb N: n> n_1, x_n\in \Gamma_1'\}$ and $B_1=:\{n\in\mathbb N: n> n_1, x_n\in \Gamma_1''\}$. If $A_1$ is infinite, then take $\Gamma_2=:\Gamma_1'$ and $n_2=:\min A_1$; otherwise, $B_1$ is infinite, thus we take $\Gamma_2=: \Gamma_1''$ and $n_2=:\min B_1$. Continuing in this manner, by the Axiom of Choice, we can produce a nested sequence $\{\Gamma_n\}$ and a subsequence $\{x_{n_k}\}$ of $\{x_n\}$ such that $x_{n_k}\in \Gamma_k$ for $k\in\mathbb N$. As well, we observe that $|\Gamma_k|=\frac{b-a}{2^{k-1}}$. By assumption $(\exists L\in\mathbb K)(\forall k\in\mathbb N)(L\in \Gamma_k)$; thus $|x_{n_k}-L|\le \frac{b-a}{2^{k-1}}$. Therefore $\{x_{n_k}\}$ converges to $L$ as $\frac{b-a}{2^{k-1}}$ converges to 0 (see remark~\ref{R: 2 Converge}).

  \item[$(v)\Rightarrow(vi)$:] Let $A\subset \mathbb{K}$ be a bounded infinite set. By the Axiom of Choice, $A$ has a denumerable subset -- that is, there exists an injection $\{x_n\}:\mathbb{N}\to A$. As $A$ is bounded, $\{x_n\}$ has a subsequence $\{x_{n_k}\}$ that converges to a point $x\in\mathbb{K}$ by assumption. Then $x$ must be a cluster point of $A$ because the sequence $\{x_{n_k}\}$ is injective, and thus not eventually constant.

  \item[$(vi)\Rightarrow(vii)$:] Let $\{x_n\}$ be a Cauchy sequence in $\mathbb{K}$. Then $\{x_n\}$ is bounded, because we can find $N\in\mathbb{N}$ such that $n\ge N$ implies that $|x_N-x_n|<1$; hence $|x_n|<1+|x_N|$ so that $\{x_n\}$ is bounded by $\max\{|x_1|,\ldots,|x_{N-1}|,|x_N|+1\}$. Thus $\mathrm{range}(\{x_n\})$ is a bounded set. If $\mathrm{range}(\{x_n\})=\{a_1,\ldots, a_k\}$ is finite, then $\{x_n\}$ is eventually constant (and thus convergent) because for sufficiently large $n,m\in\mathbb{N}$, we have $|x_n-x_m|<\min_{p\neq q}|a_p-a_q|$, where $p$ and $q$ range from $1$ to $k$. Otherwise, $\mathrm{range}(\{x_n\})$ has a cluster point $L$ by assumption. To show that $\{x_n\}\to L$, let $\epsilon\in\mathbb{K}_+$ and $N\in\mathbb{N}$ be such that $n,m\ge N$ implies that $|x_n-x_m|<\frac{\epsilon}{2}$. Observe that the set $\{n\in\mathbb{N} : |x_n-L|<\frac{\epsilon}{2} \}$ is infinite because $L$ is a cluster point (see Theorem 2.20 in Rudin~\cite{poma}), so that $A=:\{n\in\mathbb{N} : |x_n-L|<\frac{\epsilon}{2} \}\cap\{n\in \mathbb{N} : n\ge N\}$ is non-empty. Let $M=:\min A$. Then, for $n\ge N$, we have $|x_n-L|\le|x_n-x_M|+|x_M-L|<\epsilon$.
  
  \item[$(vii)\Rightarrow(viii)$:] (This proof can be found in Hewitt \& Stromberg~\cite{hewitt}, p. 44.) Let $S$ be a non-empty set bounded from above. We will construct a decreasing Cauchy sequence in $\ub( S)$ and show that the limit of the sequence is $\sup(S)$. Let $b\in\ub( S)$ and $a\in S$. Notice that there exists $M$ and $-m$ in $\mathbb{N}$ such that $m<a\le b<M$. For each $p\in\mathbb N$, we define
    \[S_p=:\left\{k\in\mathbb Z:\frac{k}{2^p}\in\ub(S)\emph{ and }k\le2^pM\right\}\]
    Clearly $2^pm$ is a lower bound of $S_p$ and $2^pM\in S_p$; hence, $S_p$ is finite which implies that $k_p=:\min S_p$ exists. We define $a_p=:\frac{k_p}{2^p}$ for all $p\in\mathbb N$. From the definition of $k_p$, it follows that $\frac{2k_p}{2^{p+1}}=\frac{k_p}{2^p}$ is an upper bound of $S$ and $\frac{2k_p-2}{2^{p+1}}=\frac{k_p-1}{2^p}$ is not. Thus, either $k_{p+1}=2k_p$ or $k_{p+1}=2k_p-1$, so that either $a_{p+1}=a_p$ or $a_{p+1}=a_p-\frac{1}{2^{p+1}}$; in either case, we have $a_{p+1}\le a_p$ and $a_p-a_{p+1}\le\frac{1}{2^{p+1}}$. Now, if $q>p\ge 1$, then
    \begin{align*}
     0\le a_p-a_q&=(a_p-a_{p+1}) +(a_{p+1} -a_{p+2})+ \cdots+(a_{q-1}-a_q)\\
     &\le \frac{1}{2^{p+1}} +\cdots +\frac{1}{2^q}
     = \frac{1}{2^{p+1}} (2-\frac{1}{2^{q-p-1}})<\frac{1}{2^p} 
    \end{align*}
    Therefore $a_p$ is a Cauchy sequence and $L=:\lim_{p\to\infty}a_p$ exists by assumption. To reach a contradiction, suppose $L\not\in\ub( S)$. Then there exists $x\in S$ such that $x>L$, and hence there exists $p\in\mathbb N$ such that $a_p-L=|a_p-L|< x-L$; thus $a_p<x$, which contradicts the fact that $a_p\in\ub( S)$. Now assume there exists $L'\in\ub( S)$ such that $L'<L$ and choose $p\in\mathbb N$ such that $\frac{1}{2^p}<L-L'$ (see remark~\ref{R: 2 Converge}). Then $a_p-\frac{1}{2^p}\ge L-\frac{1}{2^p}>L'$, thus $a_p-\frac{1}{2^p}=\frac{k_p-1}{2^p}\in\ub( S)$ which contradicts the minimality of $k_p$. Thus $L=\sup(S)$.

  \item[$(viii)\Rightarrow(ix)$:] Suppose that $\mathbb{K}$ is Dedekind complete and that $\mathbb{A}$ is a totally ordered Archimedean field extension of $\mathbb{K}$. Recall that $\mathbb{Q}$ is dense in $\mathbb{A}$ as it is Archimedean; hence, the set $\{q\in \mathbb{Q} : q<a\}$ is non-empty and bounded above in $\mathbb{K}$ for all $a\in\mathbb{A}$. Define the mapping $\sigma:\mathbb{A}\to\mathbb{K}$ by
    \[\sigma(a)=:\sup_\mathbb{K}\{q\in\mathbb{Q} : q<a\}\]
    To show that $\mathbb{A}=\mathbb{K}$ we will show that $\sigma$ is just the identity map. Note that $\sigma$ fixes $\mathbb{K}$. To reach a contradiction, suppose that $\mathbb{A}\neq\mathbb{K}$ and let $a\in\mathbb{A}\setminus\mathbb{K}$. Then $\sigma(a)\neq a$ so that either $\sigma(a)>a$ or $\sigma(a)<a$. If it is the former, then there exists $q\in\mathbb{Q}$ such that $a<q<\sigma(a)$, and if it is the latter then there exists $q\in\mathbb{Q}$ such that $\sigma(a)<q<a$ (because $\mathbb{K}$ is Archimedean by assumption so that $\mathbb{Q}$ is dense in $\mathbb{K}$). In either case we reach a contradiction. Therefore $\mathbb{K}$ has no proper Archimedean field extensions.

  \item[$(ix)\Rightarrow(i)$:] Suppose, to the contrary, there is an infinite cardinal $\kappa$ and a family $[a_i,b_i]_{i\in I}$ of fewer than $\kappa$ closed bounded intervals with the finite intersection property such that $\bigcap_{i\in I}[a_i,b_i]=\emptyset$. Let $\overline{\mathbb{K}}$ be a Dedekind complete field (see Theorem~\ref{T: exist ded}). As $\mathbb{K}$ is an Archimedean field, there is a natural embedding of $\mathbb{K}$ into $\overline{\mathbb{K}}$, so we can consider $\mathbb{K}\subseteq \overline{\mathbb{K}}$ (see Theorem~\ref{T: arch embed ded}). Because $[a_i,b_i]$ satisfies the finite intersection property, the set $A=:\{a_i :i\in I\}$ is bounded from above and non-empty so that $c=:\sup(A)$ exists in $\overline{\mathbb{K}}$, but then $a_i\le c\le b_i$ for all $i\in I$ so that $c\not\in\mathbb{K}$. Thus $\overline{\mathbb{K}}$ is a proper field extension of $\mathbb{K}$ which is Archimedean by Lemma~\ref{L: order -> arch}, a contradiction.
  \end{description}
\end{proof}

\begin{remark}
  It should be noted that the equivalence of $(ii)$ and $(vii)$ above was proved in Keisler (\cite{keisler}, pp 17-18). Also, the equivalence of $(viii)$ and $(ix)$ was proved in Banaschewski~\cite{bana} using a different method than ours that relies on the axiom of choice.
\end{remark}

\begin{remark}\label{R: NSA Unique}
  Using the Archimedean property assumed of $\mathbb K$, we can actually show that the standard and infinitesimal parts of the decomposition of a finite number are unique: let $\alpha\in\mathcal{F}(^*\mathbb{K})$ and $a,b\in\mathbb{K}$ such that $\alpha-a,\alpha-b\in\mathcal{I}(^*\mathbb{K})$, then $\alpha -a - \alpha + b=b-a\in\mathcal{I}(^*\mathbb{K})$; however, $\mathbb{K}\cap\mathcal{I}( ^*\mathbb{K})=\{0\}$ because $\mathbb K$ is Archimedean. Therefore $b=a$.
\end{remark}
\begin{remark}\label{R: 2 Converge}
  As $\mathbb K$ is Archimedean, for any $\epsilon\in\mathbb K_+$, there exists $n\in\mathbb N$ such that $\frac{1}{n}<\epsilon$. Thus, the fact that both of the sequences, $\frac{1}{n}$ and $\frac{1}{2^n}$, converge to $0$ depends only on the Archimedean property.
\end{remark}

\begin{corollary}
  If $\mathbb{K}$ is an ordered Archimedean field that is not Dedekind complete (e.g. $\mathbb{Q}$), then $^*\mathbb{K}$ contains finite numbers that are not infinitesimally close to some number of $\mathbb{K}$.
\end{corollary}
\begin{proof}
  Follows from Theorem~\ref{T: BIG ONE} as we showed $(ii)\Leftrightarrow(vii)$.
\end{proof}

In Theorem~\ref{T: BIG ONE} we showed that the nine properties listed above were equivalent under the assumption that $\mathbb K$ is Archimedean. However, just as we showed in Lemma~\ref{L: order -> arch}, we have observed that this assumption is not necessary for \emph{some} of the properties. 

To the end of this section, we show that, along with property $(viii)$, properties $(i), (ii), (iii), (v)$ and $(vi)$ imply the Archimedean property.

\begin{lemma}
  Let $\mathbb{K}$ be an ordered field. If $\mathbb{K}$ is Bolzano complete, then $\mathbb{K}$ is Archimedean.
\end{lemma}
\begin{proof}
  Suppose, to the contrary, that $\mathbb{K}$ is non-Archimedean. Then $\mathbb{N}\subset \mathbb{K}$ is bounded from above, and hence has a cluster point $L\in\mathbb{K}$. Then the set $A=:\{n\in \mathbb{N} : |L-n|<\frac{1}{2}, n\neq L\}$ is non-empty. If $A$ contains only one element $m\in\mathbb{N}$, then the set $\{n\in\mathbb{N}: |L-n|<|L-m|\}$ must be empty, which contradicts the fact that $L$ is a cluster point of $\mathbb{N}$. Otherwise, if $A$ contains two distinct elements $p,q\in\mathbb{N}$, then $|p-q|\le|p-L|+|q-L|<1$, but $p,q\in\mathbb{N}$, so this implies $p=q$, a contradiction.
\end{proof}

\begin{lemma}\label{L: bw -> mct -> arch}
  Let $\mathbb K$ be an ordered field. If $\mathbb{K}$ is either Bolzano-Weierstrass complete or monotone complete, then $\mathbb K$ is Archimedean.
\end{lemma}
\begin{proof}
  We show that Bolzano-Weiestrass completeness implies the monotone completeness, and then show that the monotone completeness implies the Archimedean property.


  Suppose that $\mathbb{K}$ is Bolzano-Weierstrass complete and let $\{x_n\}$ be a bounded monotonic sequence. Without loss of generality, we can assume that $\{x_n\}$ is increasing. Then $\{x_n\}$ has a subsequence $\{x_{n_k}\}$ that converges to some point $L\in\mathbb K$, by assumption. As $\{x_n\}$ is increasing, it follows that $\{x_{n_k}\}$ is increasing and that $L\in\ub(\{x_n\})$. Given $\epsilon\in\mathbb K_+$, we know there exists $\delta\in\mathbb K_+$ such that $(\forall k\in\mathbb N)(\delta\le k\Rightarrow|x_{n_k}-L|<\epsilon)$, but we also know that for any $m\ge n_{\delta}$, we have $x_{n_{\delta}}\le x_m\le L$. Therefore $\{x_n\}$ converges to $L$.

  Now suppose that the $\mathbb{K}$ is monotone complete. To reach a contradiction, suppose that $\mathbb K$ is a non-Archimedean. From this, it follows that the sequence $\{n\}$ is bounded and, thus, converges to some point $L$ by assumption. It should be clear that $L\in\mathcal L(\mathbb K)$. As well, we have $L-n\in\mathcal L(\mathbb K)$ for any $n\in\mathbb N$ (because $L-n\not\in\mathcal L(\mathbb K)$ implies $L<m+n\in\mathbb N$ for some $m\in\mathbb N$). However, this last condition contradicts $\{n\}$ converging to $L$, as the difference between the sequence and the point $L$ will always be infinitely large. Therefore $\mathbb K$ must be Archimedean.
\end{proof}

\begin{lemma}\label{L: gcantor -> arch}
  Let $\mathbb K$ be an ordered field. If $\mathbb{K}$ is Cantor $\kappa$-complete for $\kappa=\card(\mathbb{K})^+$, then $\mathbb K$ is Archimedean. Consequently, if $\mathbb{K}$ is Cantor $\kappa$-complete for every cardinal $\kappa$, then $\mathbb{K}$ is Archimedean.
\end{lemma}
\begin{proof}
Let $S\subset\mathbb K$ be bounded from above and $\Gamma=:\{[a,b] : a\in S, b\in\ub(S)\}$. Observe that $\Gamma$ satisfies the finite intersection property and that $\card(\Gamma)\le\card(\mathbb K)\times\card(\mathbb K)=\card(\mathbb K)$. Then there exists $\sigma\in\bigcap_{\gamma\in\Gamma} \gamma$ by assumption. Clearly $(\forall a\in S)(a\le\sigma)$, thus $\sigma\in\ub(S)$. But we also have $(\forall b\in\ub(S))(\sigma\le b)$. Therefore $\sigma=\sup(S)$.
\end{proof}

\begin{lemma}\label{L: nsa -> arch}
  Let $\mathbb K$ be a totally ordered field. If $\mathcal F(^*\mathbb K)=\mathbb K\oplus\mathcal I(^*\mathbb K)$, in the sense that every finite number can be decomposed uniquely into the sum of an element from $\mathbb K$ and an element from $\mathcal I(^*\mathbb K)$, then $\mathbb K$ is Archimedean.
\end{lemma}
\begin{proof}
  Suppose that $\mathbb K$ is non-Archimedean. Then there exists a $dx\in\mathcal I(\mathbb K)$ such that $dx\neq0$ by Proposition~\ref{P: 1 arch}. Now take $\alpha\in\mathcal F(^*\mathbb K)$ arbitrarily. By assumption there exists unique $k\in\mathbb K$ and $d\alpha\in\mathcal I(^*\mathbb K)$ such that $\alpha=k+d\alpha$. However, we know that $dx\in\mathcal I(^*\mathbb K)$ as well because $\mathbb{K}\subset{^*\mathbb{K}}$ and the ordering in $^*\mathbb{K}$ extends that of $\mathbb{K}$. Thus $(k+dx)+(d\alpha-dx)=k+d\alpha=\alpha$ where $k+dx\in\mathbb K$ and $d\alpha-dx\in\mathcal I(^*\mathbb K)$. This contradicts the uniqueness of $k$ and $d\alpha$. Therefore $\mathbb K$ is Archimedean.
\end{proof}

\section{Axioms of the Reals}\label{S: axioms}
In modern mathematics, the set of real numbers is most commonly defined in an axiomatic fashion as a totally ordered, Dedekind complete field; however, the result of Theorem~\ref{T: BIG ONE} presents multiple options for an axiomatic definition of $\mathbb{R}$. What follows are several different axiomatic definitions of the set of real numbers: the first two are based on Cantor completeness, the next three are sequential approaches, the sixth and seventh are based on properties of subsets, the eigth is based on non-standard analysis, and the last is based on an algebraic characterization.

\subsubsection*{Axioms of the Reals based on Cantor Completeness}
\begin{enumerate}[label=\textbf{C\arabic*.},ref=\textbf{C\arabic*}]
\item \label{D: R ckc}
  $\mathbb{R}$ is the set that satisfies the following axioms
  \begin{axiom-enum}
    \item $\mathbb R$ is a totally ordered field.
    \item $\mathbb R$ is Cantor $\kappa$-complete for any infinite cardinal $\kappa$ (Theorem~\ref{T: BIG ONE} number $(i)$).
  \end{axiom-enum}
\item\label{D: R cc}
  $\mathbb{R}$ is the set that satisfies the following axioms
  \begin{axiom-enum}
    \item $\mathbb R$ is a totally ordered Archimedean field.
    \item $\mathbb R$ is Cantor complete (Theorem~\ref{T: BIG ONE} number $(iv)$).
  \end{axiom-enum}
\end{enumerate}

\subsubsection*{Sequential Axioms of the Reals}

\begin{enumerate}[label=\textbf{S\arabic*.},ref=\textbf{S\arabic*}]
\item \label{D: R bw}
  $\mathbb{R}$ is the set that satisfies the following axioms
  \begin{axiom-enum}
    \item $\mathbb R$ is a totally ordered field.
    \item $\mathbb R$ is Bolzano-Weierstrass complete (Theorem~\ref{T: BIG ONE} $(v)$).
  \end{axiom-enum}
\item\label{D: R mct}
  $\mathbb{R}$ is the set that satisfies the following axioms
  \begin{axiom-enum}
    \item $\mathbb R$ is a totally ordered field.
    \item $\mathbb R$ monotone complete (Theorem~\ref{T: BIG ONE} $(iii)$).
  \end{axiom-enum}
\item\label{D: R seq}
  $\mathbb{R}$ is the set that satisfies the following axioms
  \begin{axiom-enum}
    \item $\mathbb R$ is a totally ordered Archimedean field.
    \item $\mathbb R$ is sequentially complete (Theorem~\ref{T: BIG ONE} $(vii)$).
  \end{axiom-enum}
\end{enumerate}

\subsubsection*{Set-Based Axioms of the Reals}
\begin{enumerate}[label=\textbf{B\arabic*.},ref=\textbf{B\arabic*}]
\item \label{D: R bol}
  $\mathbb{R}$ is the set that satisfies the following axioms
  \begin{axiom-enum}
    \item $\mathbb{R}$ is a totally ordered field.
    \item $\mathbb{R}$ is Bolzano complete (Theorem~\ref{T: BIG ONE} $(vi)$).
  \end{axiom-enum}
\item \label{D: R ded}
  $\mathbb{R}$ is the set that satisfies the following axioms
  \begin{axiom-enum}
    \item $\mathbb{R}$ is a totally ordered field.
    \item $\mathbb{R}$ is Dedekind complete (Theorem~\ref{T: BIG ONE} $(viii)$).
  \end{axiom-enum}
\end{enumerate}

\subsubsection*{Axioms of the Reals from Non-standard Analysis}

\begin{enumerate}[label=\textbf{N\arabic*.},ref=\textbf{N\arabic*}]
\item \label{D: R nsa}
  $\mathbb{R}$ is the set that satisfies the following axioms
  \begin{axiom-enum}
    \item $\mathbb{R}$ is a totally ordered field.
    \item $\mathbb{R}$ is Leibniz complete (Theorem~\ref{T: BIG ONE} $(ii)$ and Remark~\ref{R: NSA Unique}).
  \end{axiom-enum}
\end{enumerate}

\subsubsection*{Algebraic Axioms of the Reals}
\begin{enumerate}[label=\textbf{A\arabic*.},ref=\textbf{A\arabic*}]
\item\label{D: R alg}
  $\mathbb{R}$ is the set that satisfies the following axioms
  \begin{axiom-enum}
  \item $\mathbb{R}$ is a totally ordered Archimedean field.
  \item $\mathbb{R}$ is Hilbert complete (Theorem~\ref{T: BIG ONE} $(ix)$).
  \end{axiom-enum}
\end{enumerate}

Notice that definitions \ref{D: R alg}, \ref{D: R cc} and \ref{D: R seq} explicitly assert that $\mathbb{R}$ is an \emph{Archimedean} field, while the others do not. From the preceding section, we know that properties of the remaining definitions above (definitions \ref{D: R ckc}, \ref{D: R bw}, \ref{D: R mct}, \ref{D: R nsa}, \ref{D: R bol}, \ref{D: R ded}) are sufficient to establish the Archimedean property, but it should be clear that every non-Archimedean field is Hilbert complete, and as we will see in the next section, sequential completeness and Cantor completeness can hold in non-Archimedean fields as well.

\section{Non-standard Construction of the Real Numbers}\label{S: construct}
What we present in this section is a very quick presentation of one way to construct a Dedekind complete field using the techniques of non-standard analysis; however, this construction can not be considered a ``proof of existence of Dedekind complete fields'' because we rely on the assumption that every Archimedean field has cardinality at most $\mathfrak c$ where $\mathfrak c$ is defined in Corollary~\ref{C: ded order iso}. An alternative approach using non-standard analysis, that does not rely on the assumption that Dedekind complete fields exist, can be found in Davis~\cite{davis}, but his method is somewhat outdated as it uses the Concurrence Theorem rather than the more current concept of saturation that we employ. It is our opinion that our method can be modified to remove the assumption mentioned above, while still being simpler than the approach used in Davis.

As well, we note that a slight refinement of the construction given here can be found in Hall \& Todorov~\cite{HallTodDedekind11}.

To begin, let $^*\mathbb{Q}$ be a $\mathfrak c^+$-saturated non-standard extension of $\mathbb{Q}$ and recall that $^*\mathbb{Q}$ is a totally ordered non-Archimedean field extension of $\mathbb{Q}$.

Now define $\overline{^*\mathbb{Q}}=:\mathcal{F}(^*\mathbb{Q})/\mathcal{I}(^*\mathbb{Q})$ (see \S~\ref{S: inf} for definition of $\mathcal{F}$ and $\mathcal{I}$). Notice that $\overline{^*\mathbb{Q}}$ is a totally ordered Archimedean field since $\mathcal{F}(^*\mathbb{Q})$ is a totally ordered Archimedean ring and $\mathcal{I}(^*\mathbb{Q})$ is a maximal convex ideal in $\mathcal{F}(^*\mathbb{Q})$ (see Lemma~\ref{L: finite arch ring}). Let $q:\mathcal{F}(^*\mathbb{Q})\to\overline{^*\mathbb{Q}}$ be the cannonical homomorphism.

\begin{theorem}
  $\overline{^*\mathbb{Q}}$ is Dedekind complete.
\end{theorem}
\begin{proof}
  Let $A\subset\overline{^*\mathbb{Q}}$ be non-empty and bounded from above. If either $A$ is finite or $A\cap\ub(A)\neq\emptyset$ then we are done as $\max(A)$ exists; thus, suppose $A$ is infinite and $A\cap\ub(A)=\emptyset$. Let $B=:\ub(A)$. Then by the Axiom of Choice, there exist functions $f:A\to {^*\mathbb{Q}}$ and $g:B\to{^*\mathbb{Q}}$ such that, for all $a\in A$, $f(a)\in a$ and for all $b\in B$, $g(b)\in b$. Observe that $f(a)< g(b)$ for all $a\in A$ and $b\in B$ because $a<b$ by assumption. Consequently, the family $\{[f(a),g(b)]\}_{a\in A,b\in B}$ has the finite intersection property, and because $\overline{^*\mathbb{Q}}$ is Archimedean, we have that $\card(A\times B)=\card(\overline{^*\mathbb{Q}})\le\mathfrak c$ (see Corollary~\ref{C: arch card}). Thus, by the Saturation Principle and Theorem~\ref{T: non-arch completeness}, there exists a $\gamma\in\bigcap_{a\in A,b\in B}[f(a),g(b)]$, at which point we clearly have $\sup A=q(\gamma)$.
\end{proof}

Thus $\overline{^*\mathbb{Q}}$ is a Dedekind complete field, and from Corollary~\ref{C: ded order iso} we know that $\overline{^*\mathbb{Q}}$ is order field isomorphic to the field of Dedekind cuts $\mathbb{R}$ under the mapping from $q(\alpha)\mapsto C_\alpha$ where $C_\alpha=:\{p\in\mathbb{Q} : p<\alpha\}$ for $\alpha\in\mathcal{F}({^*\mathbb{Q}})$. What's interesting about this observation is that the sets $C_\alpha$ provide explicit form to represent a Dedekind cut using only $\mathbb{Q}$ and its non-standard extension $^*\mathbb{Q}$.

\section{Completeness of a Non-Archimedean Field}\label{S: non-arch completeness}



In this section we present some basic results concerning completeness of non-Archimedean fields. Most of the results in this section are due to H. Vernaeve~\cite{vernaeve}.


As before, $\kappa^+$ stands for the successor of $\kappa$ and $\aleph_1=\aleph_0^+$.

\begin{theorem}
  Let $\mathbb K$ be an ordered field. If $\mathbb K$ is non-Archimedean and Cantor $\kappa$-complete (see Definition~\ref{D: completeness}), then $\kappa\le\card(\mathbb K)$.
\end{theorem}
\begin{proof}
  Suppose, to the contrary, that $\kappa>\card(\mathbb K)$, then $\mathbb K$ is Cantor $\card(\mathbb K)^+$-complete. Then it follows from Lemma~\ref{L: gcantor -> arch} that $\mathbb{K}$ is Archimedean, a contradiction. 
\end{proof}

It should be noted, that in non-standard analysis there is a generalization of the following definition to what are known as \emph{internal sets}, for more information we refer to Lindstr\o m~\cite{lindstrom}.

\begin{definition}[Algebraic Saturation]\label{D: algebraic saturation}
  Let $\kappa$ be an infinite cardinal. A totally ordered field $\mathbb K$ is \emph{algebraically $\kappa$-saturated} if every family $\{(a_\gamma,b_\gamma)\}_{\gamma\in \Gamma}$ of fewer than $\kappa$ open intervals in $\mathbb K$ with the F.I.P. (finite intersection property) has a non-empty intersection, $\bigcap_{\gamma\in \Gamma} (a_\gamma,b_\gamma)\neq \emptyset$. If $\mathbb{K}$ is algebraically $\aleph_1$-saturated -- i.e. every sequence of open intervals with the F.I.P. has a non-empty intersection -- then we simply say that $\mathbb{K}$ is \emph{algebraically saturated}. As well, we say that $\mathbb K$ is \emph{algebraically $\kappa$-saturated at infinity} if every collection of fewer than $\kappa$ elements from $\mathbb K$ is bounded. Also, $\mathbb{K}$ is \emph{algebraically saturated at infinity} if $\mathbb K$ is algebraically $\aleph_1$-saturated at infinity -- i.e. every countable subset of $\mathbb{K}$ is bounded.
\end{definition}

Notice that every totally ordered field is algebraically $\aleph_0$-saturated and $\aleph_0$-saturated at infinity (in a trivial way).

\begin{theorem}\label{T: sat <-> cantor}
  Let $\mathbb K$ be an ordered field and $\kappa$ be an uncountable cardinal. Then the following are equivalent:
  \begin{thm-enumerate}
    \item $\mathbb K$ is algebraically $\kappa$-saturated 
    \item $\mathbb K$ is Cantor $\kappa$-complete and algebraically $\kappa$-saturated at infinity.
  \end{thm-enumerate}
\end{theorem}
\begin{proof}
  \mbox{}
  \begin{description}
    \item[$(i)\Rightarrow(ii)$:] Let $\mathcal C=:\{[a_{\gamma},b_{\gamma}]\}_{\gamma\in \Gamma}$ and $\mathcal O=:\{(a_{\gamma}, b_{\gamma})\}_{\gamma\in \Gamma}$ be families of fewer than $\kappa$ bounded closed and open intervals, respectively, where $\mathcal C$ has the F.I.P.. If $a_k=b_p$ for some $k,p\in\Gamma$, then  $\bigcap_{\gamma\in \Gamma}[a_{\gamma},b_{\gamma}]=\{a_k\}$ by the F.I.P. in $\mathcal C$. Otherwise, $\mathcal O$ has the F.I.P.; thus, there exists $\alpha\in\bigcap_{\gamma\in \Gamma} (a_{\gamma}, b_{\gamma})\subseteq\bigcap_{\gamma\in \Gamma}[a_{\gamma}, b_{\gamma}]$ by algebraic $\kappa$-saturation. Hence $\mathbb K$ is Cantor $\kappa$-complete. To show that $\mathbb K$ is algebraically $\kappa$-saturated at infinity, let $A\subset\mathbb K$ be a set with $\card(A)<\kappa$. Then $\bigcap_{a\in A}(a,\infty)\neq\emptyset$ by algebraic $\kappa$-saturation.

    \item[$(ii)\Rightarrow(i)$:] Let $\{(a_{\gamma},b_{\gamma})\}_{\gamma\in \Gamma}$ be a family of fewer than $\kappa$ elements with the F.I.P.. Without loss of generality, we can assume that each interval is bounded. As $\mathbb K$ is algebraically $\kappa$-saturated at infinity, there exists $\frac{1}{\rho} \in \ub(\{ \frac{1}{b_l -a_k} : l, k\in \Gamma \})$ (that is, $\frac{1}{b_l-a_k}\le\frac{1}{\rho}$ for all $l,k\in \Gamma$) which implies that $\rho>0$ and that $\rho$ is a lower bound of $\{b_l-a_k : l, k\in \Gamma\}$. Next, we show that the family $\{[a_\gamma+\frac{\rho}{2},b_\gamma-\frac{\rho}{2}]\}_{\gamma\in\Gamma}$ satisfies the F.I.P.. Let $\gamma_1,\ldots,\gamma_n\in\Gamma$ and $\zeta=:\max_{k\le n}\{a_{\gamma_k} + \frac{\rho}{2}\}$. Then, for all $m\in\mathbb N$ such that $m\le n$, we have $a_{\gamma_m} + \frac{\rho}{2}\le \zeta \le b_{\gamma_m} - \frac{\rho}{2}$ by the definition of $\rho$; thus, $\zeta\in[a_{\gamma_m}+\frac{\rho}{2}, b_{\gamma_m}-\frac{\rho}{2}]$ for $m\le n$. By Cantor $\kappa$-completeness, there exists $\alpha\in\bigcap_{\gamma\in\Gamma} [a_{\gamma}+\frac{\rho}{2}, b_{\gamma}-\frac{\rho}{2}]\subseteq\bigcap_{\gamma\in\Gamma}(a_{\gamma},b_{\gamma})$.
  \end{description}
\end{proof}

\begin{corollary}\label{C: sat -> seq}
  Let $\mathbb K$ be an ordered field. If $\mathbb K$ is algebraically saturated, then every convergent sequence is eventually constant. Consequently, $\mathbb K$ is sequentially complete.
\end{corollary}
\begin{proof}
  Let $x_n\to L$ and assume that $\{x_n\}$ is not eventually constant. Then there exists a subsequence $\{x_{n_k}\}$ such that $\delta_k=:|x_{n_k}-L|>0$ for all $k\in\mathbb N$. Thus, there exists $\epsilon\in\bigcap_{k\in\mathbb N} (0,\delta_k)$ by assumption; hence $0<\epsilon<\delta_k$ for all $k\in\mathbb K$, which contradicts $\delta_k\to0$. Finally, suppose $\{x_n\}$ is a Cauchy sequence and observe that $|x_{n+1}-x_n|\to0$. Thus, by what we just proved, $|x_{n+1}-x_n|=0$ for sufficiently large $n\in\mathbb N$. Hence $\{x_n\}$ is eventually constant.
\end{proof}

\begin{corollary}\label{C: cantor -> sequential}
  Let $\mathbb K$ be an ordered field. If $\mathbb K$ is Cantor complete, but not algebraically saturated, then:
  \begin{thm-enumerate}
  \item $\mathbb{K}$ has an increasing unbounded sequence.
  \item $\mathbb{K}$ is sequentially complete.
  \end{thm-enumerate}
\end{corollary}
\begin{proof}
  \mbox{}
  \begin{thm-enumerate}
  \item By Theorem~\ref{T: sat <-> cantor} we know that $\mathbb K$ has a subset $A\subset\mathbb K$ that is unbounded. Let $x_1\in A$ be arbitrary. Now assume that $x_n$ has been defined, then there exists $c\in A$ such that $x_n<c$ by assumption; define $x_{n+1}=:c$. Using this inductive definition (and the axiom of choice), we find that $\{x_n\}$ is an increasing unbounded sequence in $\mathbb K$.
  \item By Part $(i)$ of this corollary, we know there exists an unbounded increasing sequence $\{\frac{1}{\epsilon_n}\}$. We observe that $\{\epsilon_n\}$ is a decreasing never-zero sequence that converges to zero. Let $\{x_n\}$ be a Cauchy sequence in $\mathbb K$. For all $n\in\mathbb N$, we define $S_n=:[x_{m_n} - \epsilon_n, x_{m_n} + \epsilon_n]$, where $m_n=:\min\{k\in\mathbb N : (\forall l,j\in\mathbb N)(k\le l,j\Rightarrow |x_l-x_j|<\epsilon_n)\}$ (which exists as $\{x_n\}$ is a Cauchy sequence). To show that the family $\{S_n\}_{n\in\mathbb{N}}$ satisfies the finite intersection property, let $A\subset \mathbb N$ be finite and $\rho=:\max(A)$; then we observe that $x_{m_{\rho}}\in S_k$ for any $k\in A$ because $m_k\le m_{\rho}$. Therefore there exists $L\in\bigcap_{k=1}^{\infty}S_k$ by Cantor completeness. To show that $x_n\to L$, we first observe that, given any $\delta\in\mathbb K_+$, we can find an $n\in\mathbb N$ such that $2\epsilon_n<\delta$ because $\{\epsilon_n\}$ converges to zero. As well, we note that $L\in S_n$ and that the width of $S_n$ is $2\epsilon_n$ for all $n\in\mathbb N$. Thus, given $\delta\in\mathbb K_+$ we can find $n\in\mathbb N$ such that $2\epsilon_n<\delta$, and, because $(\forall l\in\mathbb N)(m_n\le l\Rightarrow x_l\in S_n)$, we have $(\forall l\in\mathbb N)(m_n\le l\Rightarrow |L-x_l|<2\epsilon_n<\delta)$.
  \end{thm-enumerate}
\end{proof}

The previous two corollaries can be summarized in the following result.

\begin{corollary}\label{T: non-arch completeness}
  Let $\mathbb K$ be an ordered field, then we have the following implications
  $$\mathbb K\textrm{ is }\kappa\textrm{-saturated}\implies\mathbb K\textrm{ is Cantor }\kappa\textrm{-complete}\implies \mathbb K\textrm{ is sequentially complete}$$
\end{corollary}
\begin{proof}
  The first implication follows from Theorem~\ref{T: sat <-> cantor}. For the second implication we have two cases: $\mathbb{K}$ is algebraically saturated, or $\mathbb{K}$ is not-algebraically saturated. For the first case we have Corollary~\ref{C: sat -> seq} and for the second we have Corollary~\ref{C: cantor -> sequential}. 
\end{proof}


\section{Valuation Fields}\label{S: val}

Before we begin with the examples, see the next section, we quickly define \emph{ordered valuation fields} for our use later; for readers interested in the subject we refer to P. Ribenboim~\cite{riben}, A.H. Lightstone \& A. Robinson~\cite{lightstone}, and Todorov~\cite{todor-inf}. In what follows, let $\mathbb K$ be an ordered field.
 
\begin{definition}[Ordered Valuation Field]
  The mapping $v:\mathbb K\to\mathbb R\cup\{\infty\}$ is called a \emph{non-Archimedean valuation} on $\mathbb K$ if, for every $x,y\in\mathbb K$,
  \begin{def-enumerate}
    \item $v(x)=\infty$ \ifff $x=0$
    \item $v(xy)=v(x)+v(y)$ (Logarithmic property)
    \item $v(x + y)\ge\min\{v(x), v(y)\}$ (Non-Archimedean property)
    \item $|x| < |y|$ implies $v(x) \ge v(y)$ (Convexity property)
  \end{def-enumerate}
  The structure $(\mathbb K,v)$ is called an \emph{ordered valuation field}. As well, a valuation $v$ is \emph{trivial} if $v(x)=0$ for $x\in\mathbb K\setminus\{0\}$; otherwise, $v$ is \emph{non-trivial}.
\end{definition}
 
\begin{remark}[Krull's Valuation]
  It is worth noting that the definition given above can be considered as a specialized version of that given by Krull: a valuation is a mapping $v:\mathbb K\to G\cup\{\infty\}$ where $G$ is an ordered abelian group.
\end{remark}

\section{Examples of Sequentially and Spherically Complete Fields}\label{S: examples 1}

Although they may not be as well known as Archimedean fields,  non-Archimedean fields have been used in a variety of different settings: most notably to produce non-standard models in the model theory of fields and to provide a field for non-standard analysis. In an effort to make non-Archimedean fields seem less exotic than the may initially appear, we have compiled a very modest list of non-Archimedean fields that have been used in different areas of mathematics. Our first couple of examples are of fields of formal power series, which should be more familiar to the reader, as the use of these fields as generalized scalars in analysis predates A. Robinson's work in non-standard analysis \cite{todor-asymp} and are used quite often as non-standard models in model theory. For more information on these fields, we direct the reader to D. Laugwitz~\cite{laugwitz} and Todorov \& Wolf~\cite{todor-wolf}. The last of the examples are based on A. Robinson's theory of non-standard extensions. The key distinction that we would like to emphasize between these two forms of non-Archimedean fields, is that the fields of power series are at most spherically complete while the fields from non-standard analysis are always Cantor complete if not algebraically saturated.

In what follows, $\mathbb{K}$ is a totally ordered field. As well, in the next four examples we present sets of series for which we assume have been supplied with the normal operations of \emph{polynomial-like} addition and multiplication. Under this assumption, each of the following sets are in fact fields. 

\begin{example}[Hahn Series]
  Let $\supp(f)$ denote the \emph{support} of $f$ (i.e. values in the domain for which $f$ is non-zero). The field of \emph{Hahn series} is defined to be the set $$\mathbb K((t^{\mathbb R}))=: \left\{\sum_{r\in\mathbb R} a_rt^r : a_r\in\mathbb K\emph{ and }\supp(r\mapsto a_r)\subseteq\mathbb R\emph{ is well ordered}\right\}$$ which can be supplied with the \emph{canonical valuation} $\nu:\mathbb K((t^{\mathbb R}))\to\mathbb R\cup\{\infty\}$ defined by $\nu(0)=:\infty$ and $\nu(A)=:\min(\supp(r\mapsto a_r))$ for all $A\in\mathbb K((t^{\mathbb R}))$. As well, $\mathbb K((t^{\mathbb R}))$ has a natural ordering given by
  \[\mathbb K((t^{\mathbb R}))_+=:\left\{A=\sum_{r\in\mathbb R} a_rt^r \in\mathbb K((t^{\mathbb R})) : a_{\nu(A)}>0\right\}\]
\end{example}

\begin{remark}
  For our purposes here, the preceding definition of the field of Hahn series is sufficient; however, there is a more general definition in which the additive group $\mathbb{R}$ is replaced with an abelian ordered group $G$.
\end{remark}

\begin{example}[Levi-Civita]
  The field of Levi-Civita series is defined to be the set $$\mathbb K\langle t^{\mathbb R}\rangle=:\left\{\sum_{n=0}^{\infty}a_nt^{r_n} : a_n\in\mathbb K \emph{ and }\{r_n\}\emph{ is strictly increasing and unbounded in }\mathbb R\right\}.$$ As $\{r_n : n\in\mathbb N\}$ is well-ordered whenever $\{r_n\}$ is strictly increasing, we can embed $\mathbb K\langle t^{\mathbb R}\rangle$ into $\mathbb K((t^{\mathbb R}))$ in an obvious way: $\sum_{n=0}^{\infty}a_nt^{r_n}\mapsto\sum_{k\in\mathbb R}\beta_kt^k$ where $\beta_{r_n}=:a_n$ for all $n\in\mathbb N$, and $\beta_k=:0$ for $k\not\in\textrm{range}(\{r_n\})$. Thus, $\mathbb K\langle t^{\mathbb R}\rangle$ can be ordered by the ordering inherited from $\mathbb K((t^{\mathbb R}))$.
\end{example}

\begin{example}[Laurent Series]
  The field 
  $$\mathbb K(t^{\mathbb Z})=:\left\{\sum_{n=m}^{\infty}a_nt^n : a_n\in\mathbb K \emph{ and }m\in\mathbb Z\right\}$$ of formal \emph{Laurent series} have a very simple embedding into $\mathbb K\langle t^{\mathbb R}\rangle$, and, thus, $\mathbb K(t^{\mathbb Z})$ is an ordered field using the ordering inherited from $\mathbb K\langle t^{\mathbb R}\rangle$.
\end{example}

\begin{example}[Rational Functions]
  The field 
  $$\mathbb K(t)=:\left\{\frac{P(t)}{Q(t)} : P, Q\in\mathbb K[t]\emph{ and }Q\not\equiv0\right\}$$ of \emph{rational functions} can be embedded into $\mathbb K(t^{\mathbb Z})$ by associating each rational function with its Laurent expansion about zero; thus $\mathbb K(t)$ inherits the ordering from $\mathbb K(t^{\mathbb Z})$. As well, we can consider $\mathbb K\subset\mathbb K(t)$ under the canonical embedding given by $a\mapsto f_a$ where $f_a(t)\equiv a$, for all $a\in\mathbb K$.
\end{example}

As we mention above, each field that we defined in examples 1-4 can be structured as such: $$\mathbb K\subset \mathbb K(t)\subset \mathbb K(t^{\mathbb Z})\subset \mathbb K\langle t^{\mathbb R}\rangle\subset \mathbb K((t^{\mathbb R}))$$ What we claim, is that all of these extensions of $\mathbb K$ are in fact non-Archimedean (regardless of whether $\mathbb K$ is Archimedean or not). But this is simple: because each field is a subfield of the following fields, all we have to show is that $\mathbb K(t)$ is non-Archimedean, and this just follows from the observation that $\sum_{n=1}^{\infty}t^n=\frac{t}{1-t}\in\mathbb K(t)$ is a non-zero infinitesimal and $\sum_{n=-1}^{\infty} t^n=\frac{1}{t-t^2}\in\mathbb K(t)$ is infinitely large -- that is, $0<\frac{t}{1-t}<\frac{1}{n}$ and $\frac{1}{t-t^2}>n$ for all $n\in\mathbb{N}$.

\begin{theorem}\label{T: series real-closed}
  If $\mathbb K$ is real closed, then both $\mathbb K\langle t^{\mathbb R}\rangle$ and $\mathbb K((t^{\mathbb R}))$ are real closed.
\end{theorem}
\begin{proof} 
  See Prestel~\cite{prestel}.
\end{proof}

\begin{definition}[Valuation Metric]
  Let $(\mathbb{K},v)$ be an ordered valuation field, then the mapping $d_v:\mathbb{K}\times\mathbb{K}\to\mathbb{R}$ given by $d_v(x,y)=e^{-v(x-y)}$, where $e^{-\infty}=0$, is the \emph{valuation metric} on $(\mathbb{K},v)$. We denote by $(\mathbb{K},d_v)$ the corresponding metric space. Further, if $c\in\mathbb{K}$ and $r\in\mathbb{R}_+$, then we define the corresponding sets of open and closed balls by
  \begin{align*}
    B(c, r)=:\{k\in\mathbb{K} : d_v(c,k)<r\}\\
    \overline{B}(c,r)=:\{k\in\mathbb{K} : d_v(c,k)\le r\}
  \end{align*}
  respectively.
\end{definition}

\begin{definition}[Spherically Complete]
  A metric space is \emph{spherically complete} if every nested sequence of closed balls has a non-empty intersection.
\end{definition}

From the definition, it should be clear that every metric space is spherically complete is sequentially complete.

\begin{example}
  The metric space $(\mathbb{R}((t^{\mathbb{R}})),d_v)$ is spherically complete. A proof of this can be found in W. Krull~\cite{krull} and Theorem~2.12 of W.A.J. Luxemburg~\cite{luxemburg-valuation}.
\end{example}

\begin{example}
  Both $\mathbb{R}(t^{\mathbb Z})$ and $\mathbb{R}\langle t^{\mathbb R}\rangle$ are sequentially complete fields. A proof of this can be found in D. Laugwitz~\cite{laugwitz}.
\end{example}


\section{Examples of Cantor Complete and Saturated Fields}\label{S: examples 2}
All the examples given from this point on are based on the non-standard analysis of A. Robinson; for an introduction to the area we refer the reader to either Robinson~\cite{robinson1} \cite{robinson2}, Luxemburg~\cite{luxemburg}, Davis~\cite{davis}, Lindstr\o m~\cite{lindstrom} or Cavalcante~\cite{cavalcante}. For axiomatic introductions in particular, we refer the reader to Lindstr\o m~\cite{lindstrom} (p. 81-83) and Todorov~\cite{todor-axiomatic} (p. 685-688).

As we remarked in the beginning of this section, what is so particularly interesting about these fields is that they are all Cantor complete; a property that none of the fields of formal power series with real coefficients share (i.e. the fields given by replacing $\mathbb{K}$ with $\mathbb{R}$).

\begin{example}
  Let $\kappa$ be an infinite cardinal and $^*\mathbb R$ be a $\kappa$-saturated non-standard extension of $\mathbb R$ (see Lindstr\o m~\cite{lindstrom}). Then $^*\mathbb R$ is a non-Archimedean, real closed, algebraically $\kappa$-saturated field (see Definition~\ref{D: algebraic saturation}) with $\mathbb R\subset{^*\mathbb R}$. It follows from Theorem~\ref{T: sat <-> cantor} that $^*\mathbb R$ is Cantor $\kappa$-complete.
\end{example}

\begin{definition}[Convex Subring]
  Let $F$ be an ordered ring and $R\subset F$ be a ring. Then $R$ is a \emph{convex subring}, if for all $x\in F$ and $y\in R$, we have that $0\le|x|\le|y|$ implies that $x\in R$. Similarily, an ideal $I$ in $R$ is called convex if, for every $x\in R$ and $y\in I$, $0\le|x|\le|y|$ implies that $x\in I$.
\end{definition}


\begin{example}[Robinson's Asymptotic Field]\label{E: robinson asymptotic}
  Let $^*\mathbb R$ be a non-standard extension of $\mathbb R$ and $\rho$ be a positive infinitesimal in $^*\mathbb{R}$. We define the sets of non-standard $\rho$-\emph{moderate} and $\rho$-\emph{negligible} numbers to be 
  \begin{eqnarray*}
    \mathcal M_{\rho}(^*\mathbb R)&=:&\{\zeta\in{^*\mathbb R} : |\zeta|\le \rho^{-m} \emph{ for some } m\in\mathbb N\}\\
    \mathcal N_{\rho}(^*\mathbb R)&=:&\{\zeta\in{^*\mathbb R} : |\zeta| < \rho^n \emph{ for all } n\in\mathbb N\}
  \end{eqnarray*}
  respectively. The \emph{Robinson field of real $\rho$-asymptotic numbers} (A. Robinson~\cite{robinson1},\cite{robinson2}) is the factor ring ${^{\rho}\mathbb R}=:\mathcal M_{\rho} / \mathcal N_{\rho}$. As it is not hard to show that $\mathcal M_{\rho}$ is a convex subring, and $\mathcal N_{\rho}$ is a maximal convex ideal, it follows that $^{\rho}\mathbb R$ is an orderable field. From Todorov \& Vernaeve~\cite{todor} (Theorem 7.3, p. 228) we know that $^\rho\mathbb{R}$ is real-closed. As well, $^{\rho}\mathbb R$ is not algebraically saturated as we observe that the sequence $\{\rho^{-n}\}_{n\in\mathbb N}$ is unbounded and increasing (see Corollary~\ref{C: cantor -> sequential}). Following from the preceding observation and the fact that $^{\rho}\mathbb R$ is Cantor complete (Todorov \& Vernaeve~\cite{todor-asymp}, Theorem 10.2, p. 24), we can apply Lemma~\ref{C: cantor -> sequential} to find that $^{\rho}\mathbb R$ is sequentially complete. 
\end{example}

There are several interesting properties of $^\rho\mathbb{R}$ that distinguish it from other fields: it is non-Archimedean; it is not algebraically saturated at infinity (see Definition~\ref{D: algebraic saturation}); it has a countable topological basis (the sequence of intervals $(-s^n, s^n)$, where $s$ is the image of $\rho$ under the quotient mapping from $\mathcal{M}_\rho$ to $^\rho\mathbb{R}$, forms a basis for the neighborhoods of 0). As well, the field of Hahn series $\mathbb{R}((t^\mathbb{R}))$ can be embedded into $^\rho\mathbb{R}$ (see Todorov \& Wolf~\cite{todor-wolf}) by mapping a Hahn series $\sum_{r\in\mathbb{R}} a_rt^r$ to the series $\sum_{r\in\mathbb{R}} a_rs^r$ in $^\rho\mathbb{R}$ -- which converges in $^\rho\mathbb{R}$, but not in $^*\mathbb{R}$! To summarize this fact, we present the following chain of inclusions that extends upon that given in the preceding subsection.

\[\mathbb R\subset \mathbb R(t)\subset \mathbb R(t^{\mathbb Z})\subset \mathbb R\langle t^{\mathbb R}\rangle\subset \mathbb R((t^{\mathbb R}))\subset {^\rho\mathbb{R}}\]

Thus ${^{\rho}\mathbb R}$ is a totally ordered, non-Archimedean, real closed, sequentially complete, and Cantor complete field that is not algebraically saturated and contains all of the fields of formal power series listed above.


\begin{example}
  Let $\mathcal M\subset{^*\mathbb R}$ be a convex subring of $^*\mathbb R$ and let $\mathcal I_{\mathcal M}$ be the set of non-invertible elements of $\mathcal M$. Then $\widehat{\mathcal M}=:\mathcal M/\mathcal I_{\mathcal M}$ is a real closed, Cantor complete field. In the case that $\mathcal M=\mathcal F(^*\mathbb R)$, then $\widehat{\mathcal M}=\mathbb R$. Otherwise, $\widehat{\mathcal M}$ is non-Archimedean. These fields, $\widehat{\mathcal M}$ are referred to as $\mathcal M$-asymptotic fields. It should also be noted that for some $\mathcal M$, it might be that $\widehat{\mathcal M}$ is saturated. For more discussion about these fields we refer to Todorov~\cite{todor-lecturenotes}. Notice that $^\rho\mathbb{R}$ is a particular $\mathcal{M}$-asymptotic field that is given by $\mathcal{M}=\mathcal{M}_\rho$.
\end{example}

\begin{remark}
  It may be worth noting that, in $^\rho\mathbb{R}$, the metric topology generated from the valuation metric, and the order topology are equivalent.
\end{remark}

\end{document}